\documentclass[onefignum,onetabnum]{siamart190516}
\usepackage{lipsum}
\usepackage{amsfonts}
\usepackage{graphicx}
\usepackage{epstopdf}
\usepackage{booktabs}
\usepackage{mathtools}

\usepackage{algpseudocode}
\algnewcommand{\Initialize}[1]{
  \State \textbf{Initialize:}
  \Statex \hspace*{\algorithmicindent}\parbox[t]{.8\linewidth}{\raggedright #1}
}
\algnewcommand{\Indent}[2]{
  \State {#1}
  \vspace{-2mm}
  \Statex \hspace*{\algorithmicindent}\parbox[t]{.9\linewidth}{\raggedright #2}
}
\renewcommand{\Comment}[2][.5\linewidth]{%
  \leavevmode\hfill\makebox[#1][l]{$\triangleright$~#2}}

\ifpdf
  \DeclareGraphicsExtensions{.eps,.pdf,.png,.jpg}
\else
  \DeclareGraphicsExtensions{.eps}
\fi

\usepackage{enumitem}

\usepackage[caption=false]{subfig}
\usepackage{amssymb}
\newcommand{\vect}[1]{\boldsymbol{#1}}
\newcommand{\vectt}[1]{\boldsymbol{\mathbf{#1}}}
\newcommand{\dx}{\mathrm{d}x}
\newcommand{\divv}[1]{\nabla \cdot #1}
\newcommand{\lsb}{[\![}
\newcommand{\rsb}{]\!]}
\newcommand{\lcb}{\{\!\!\{}
\newcommand{\rcb}{\}\!\!\}}
\newcommand{\DG}{\mathrm{DG}}
\newcommand{\tens}[1]{\pmb{\mathsf{#1}}}
\newcommand{\Hdiv}{\vect{H}(\mathrm{div}; \Omega)}
\newcommand{\Hcurl}{\vect{H}(\mathrm{curl}; \Omega)}

\newcommand{\tA}{\tens{A}}
\newcommand{\tB}{\tens{B}}
\newcommand{\tC}{\tens{C}}
\newcommand{\tD}{\tens{D}}
\newcommand{\tE}{\tens{E}}
\newcommand{\tH}{\tens{H}}
\newcommand{\tS}{\tens{S}}
\newcommand{\tM}{\tens{M}}
\newcommand{\tzero}{\tens{0}}

\newcommand{\vf}{\vectt{f}}

\usepackage{standalone}
\usepackage{tikz}
\usetikzlibrary{trees}
\definecolor{babyblue}{rgb}{0.54, 0.81, 0.94}
\definecolor{classicrose}{rgb}{0.98, 0.8, 0.91}

\newsiamremark{remark}{Remark}
\newsiamremark{hypothesis}{Hypothesis}
\crefname{hypothesis}{Hypothesis}{Hypotheses}
\newsiamthm{claim}{Claim}

\headers{Multiple 3D solutions in topology optimization}{I.~P.~A.~Papadopoulos and P.~E.~Farrell}

\title{Preconditioners for computing multiple solutions in three-dimensional fluid topology optimization \thanks{Submitted DATE.
\funding{This manuscript was prepared whilst the first author was at the University of Oxford and was supported by the EPSRC Centre for Doctoral Training in Partial Differential Equations: Analysis and Applications [grant number  EP/L015811/1] and The MathWorks, Inc. Revisions were supported by the EPSRC grant ``Spectral element methods for fractional differential equations, with applications in applied analysis and medical imaging" [grant number EP/T022132/1] and the Leverhulme Trust Research Project ``Constructive approximation theory on and inside algebraic curves and surfaces" [grant number RPG-2019-144]. The second author is supported by the Engineering and Physical Sciences Research Council [grant numbers EP/R029423/1 and EP/W026163/1].}}}

\author{Ioannis P.~A.~Papadopoulos\thanks{Department of Mathematics, Imperial College London, London, UK \newline  \hspace*{4mm} (\email{ioannis.papadopoulos13@imperial.ac.uk}).}
\and Patrick E.~Farrell\thanks{Mathematical Institute, University of Oxford, Oxford, UK (\email{patrick.farrell@maths.ox.ac.uk}).}}

\usepackage{amsopn}

\ifpdf
\hypersetup{
  pdftitle={Multiple 3D solutions in topology optimization},
  pdfauthor={I. P. A. Papadopoulos, and P. E. Farrell}
}
\fi

\begin{document}

\maketitle

\begin{abstract}
Topology optimization problems generally support multiple local minima, and real-world applications are typically three-dimensional. In previous work [I.~P.~A.~Papadopoulos, P.~E.~Farrell, and T.~M.~Surowiec, \emph{Computing multiple solutions of topology optimization problems}, SIAM Journal on Scientific Computing, (2021)], the authors developed the deflated barrier method, an algorithm that can systematically compute multiple solutions of topology optimization problems. In this work we develop preconditioners for the linear systems arising in the application of this method to Stokes flow, making it practical for use in three dimensions. In particular, we develop a nested block preconditioning approach which reduces the linear systems to solving two symmetric positive-definite matrices and an augmented momentum block. An augmented Lagrangian term is used to control the innermost Schur complement and we apply a geometric multigrid method with a kernel-capturing relaxation method for the augmented momentum block. We present multiple solutions in three-dimensional examples computed using the proposed iterative solver.
\end{abstract}

\begin{keywords}
topology optimization, multiple solutions, deflated barrier method, preconditioning, multigrid
\end{keywords}

\begin{AMS}
65F08, 65K10, 65N55, 35Q35, 49M15, 90C26
\end{AMS}

\section{Introduction}
Topology optimization has become a prominent tool in engineering design.
The objective is to find the optimal distribution of a continuum within a given domain that minimizes a problem-specific cost functional, without prior knowledge of the shape or topology of the solution \cite{Bendsoe2004}.

Topology optimization problems are generally nonconvex and can support multiple local minima. A major challenge in topology optimization is to identify multiple local minimizers, so that the best one (in performance, manufacturability, or aesthetics) may be chosen. Continuation in model parameters is often used to promote convergence to ``better'' local minima. However, Stolpe and Svanberg \cite{Stolpe2001} show that such methods can be ineffective even on the simplest of topology optimization models.

In recent work, Papadopoulos et al.~\cite{Papadopoulos2021a} developed an algorithm, called the deflated barrier method, that can systematically discover \emph{multiple} stationary points of topology optimization problems formulated using a density approach.
The deflated barrier method combines the techniques of barrier methods \cite{Fiacco1990, Forsgren2002, Frisch1955, Schiela2007, Schiela2008, Ulbrich2009, Weiser2008}, primal-dual active set solvers \cite{Benson2003}, and deflation \cite{Farrell2015, Farrell2019}. In one example \cite[Fig.~1]{Papadopoulos2021a}, the authors found 42 stationary points of a fluid topology optimization problem constrained by the Navier--Stokes equations in a rectangular domain with five small holes removed. However, an application of the deflated barrier method requires the solution of many linear systems similar to those solved in an all-at-once (simultaneous analysis and design, SAND) approach. In~\cite{Papadopoulos2021a}, a sparse LU factorization was used to solve these linear systems, which is not generally affordable for three-dimensional problems.

In this work, we develop preconditioners for the linear systems arising in the deflated barrier method when applied to the model proposed by Borrvall and Petersson \cite{Borrvall2003}. The goal of this model is to minimize the power dissipation of a fluid in Stokes flow, subject to a volume constraint restricting the proportion of the domain that the fluid can occupy. The preconditioners we develop make it feasible to identify multiple solutions to the Borrvall--Petersson problem in three dimensions with the deflated barrier method.
A number of strategies have been proposed for the solution of three-dimensional topology optimization of fluid flow in other contexts \cite{Aage2013, Aage2007,Alexandersen2020,Behrou2019,Challis2009, Deng2013,Evgrafov2014,Jensen2018,Pingen2007,Sa2016,Yaji2014,Zhou2008}.

We choose a discontinuous Galerkin $H(\mathrm{div})$-conforming finite element discretization for the velocity and pressure pair \cite{Brezzi1987,Brezzi1985} due to its pressure-robustness \cite{John2017} and its crisp characterization of the kernel of the divergence operator. We will show that block preconditioning can reduce the linear systems arising in the deflated barrier method to ones that resemble the systems arising in the discretization of the Stokes--Brinkman equations \cite{Evgrafov2014}. Then, we apply modern block preconditioning, pioneered by Wathen and coworkers \cite{Elman2014, Stoll2009, Wathen2015} to solve the inner linear systems. The inner linear systems are themselves solved with block preconditioning, using an augmented Lagrangian term to control the innermost Schur complement \cite{Fortin1983}. Finally, we develop a geometric multigrid method for the augmented momentum block with a vertex-star patch relaxation that captures the kernel of the augmented Lagrangian term \cite{pcpatch, Farrell2019a, Hong2016,Schoberl1999}. The multigrid scheme also requires a characterization of the active set on all levels of the mesh hierarchy, which we discuss.

We introduce the Borrvall--Petersson model in \cref{sec:BP} and the deflated barrier method in \cref{sec:dab}. In \cref{sec:discretization} we detail the discontinuous Galerkin finite element discretization and derive the linear systems that arise in the subproblems of the deflated barrier method. In \cref{sec:preconditioning} we develop the preconditioning strategies for these linear systems. In \cref{sec:examples} we investigate a number of two- and three-dimensional examples. All the examples support multiple solutions which are discovered by the deflated barrier method, and solved with our proposed iterative methods.

\section{Topology optimization of Stokes flow}
\label{sec:BP}
Given a volume constraint on a fluid in a fixed Lipschitz domain $\Omega \subset \mathbb{R}^d$, $d \in \{2, 3\}$, the Borrvall--Petersson model attempts to minimize the energy lost by the flow due to viscous dissipation, whilst maximizing the flow velocities at the applied body force. More precisely, the objective is to find $(\vect{u},\rho) \in H^1_{\vect{g},\mathrm{div}}(\Omega)^d \times C_\gamma$ that minimizes
\begin{align}
J(\vect{u},\rho) \coloneqq\frac{1}{2} \int_\Omega \left(\alpha(\rho) |\vect{u}|^2 + \nu |\nabla \vect{u}|^2 - 2\vect{f} \cdot \vect{u}\right)  \text{d}x,\label{borrvallmin} \tag{BP}
\end{align}
where $\vect{u}: \Omega \to \mathbb{R}^d$ denotes the velocity of the fluid, $\rho: \Omega \to \mathbb{R}$ is the material distribution of the fluid and
\begin{align*}
H^1_{\vect{g}}(\Omega)^d &\coloneqq \{\vect{v} \in H^1(\Omega)^d: \vect{v}|_{\partial \Omega} = \vect{g} \},\\
H^1_{\vect{g},\mathrm{div}}(\Omega)^d &\coloneqq \{\vect{v} \in H^1_{\vect{g}}(\Omega)^d : \mathrm{div}(\vect{v}) = 0 \; \text{a.e.\ in} \; \Omega \},\\
C_{\gamma} &\coloneqq \left\{ \eta \in L^\infty(\Omega) : 0 \leq \eta \leq 1 \; \text{a.e.}, \;\; \int_\Omega \eta \; \text{d}x \leq \gamma |\Omega| \right \}.
\end{align*}
Here, $H^1(\Omega)$ denotes the Sobolev space $W^{1,2}(\Omega)$, equipped with the inner product
\begin{align}
(u, v)_{H^1(\Omega)} = \int_{\Omega}uv + \nabla u \cdot \nabla v \, \dx,
\end{align}
which induces the norm $\| v \|_{H^1(\Omega)} \coloneqq (v,v)_{H^1(\Omega)}^{1/2}$. $L^\infty(\Omega)$ denotes the vector space of essentially bounded measurable functions equipped with the essential supremum norm, i.e.~$\|\eta\|_{L^\infty(\Omega)} \coloneqq \inf\{c \geq 0 : |\eta(x)| \leq c \; \text{for a.e.} \; x \in \Omega\}$. Furthermore, $\vect{f} \in L^2(\Omega)^d$ is a body force, $\nu > 0$ is the (constant) viscosity, and $\gamma \in (0,1)$ is the volume fraction. The restriction, $|_{\partial \Omega}$, is to be understood in the boundary trace sense \cite[Ch.~5.5]{Evans2010}. Moreover, the boundary data $\vect{g} \in H^{1/2}(\partial \Omega)^d$ and $\vect{g} = \vect{0} $ on $\Gamma \subset \partial \Omega$, with $\mathcal{H}^{d-1}(\Gamma)>0$, i.e.\ $\Gamma$ has nonzero Hausdorff measure on the boundary. Borrvall and Petersson introduced the inverse permeability term, $\alpha$, which models the influence of the material distribution on the flow. For values of $\rho$ close to one, $\alpha(\rho)$ is small, permitting fluid flow; for values of $\rho$ close to zero, $\alpha(\rho)$ is very large, restricting fluid flow. The function $\alpha$ satisfies the following properties:
\begin{enumerate}[label=({A}\arabic*)]
\item $\alpha: [0,1] \to [\underline{\alpha}, \overline{\alpha}]$ with $0 \leq \underline{\alpha} < \overline{\alpha} < \infty$;
\label{alpha1}
\item $\alpha$ is monotonically decreasing; \label{alpha2}
\item $\alpha(0) = \overline{\alpha}$ and $\alpha(1) = \underline{\alpha}$;
\label{alpha3}
\item $\alpha$ is twice continuously differentiable, \label{alpha4}
\item $\alpha$ is strongly convex, \label{alpha5}
\end{enumerate}
generating an operator also denoted $\alpha: C_\gamma \to L^\infty(\Omega; [\underline{\alpha},\overline{\alpha}])$. Typically, in the literature, $\alpha$ takes the form \cite{Borrvall2003, Evgrafov2014}
\begin{align}
\alpha(\rho) = \bar{\alpha}\left( 1 - \frac{\rho(q+1)}{\rho+q}\right), \label{eq:alphachoice}
\end{align}
where $q>0$ is a penalty parameter, so that $\lim_{q \to \infty} \alpha(\rho) = \bar{\alpha}(1-\rho)$. Borrvall and Petersson \cite[Sec.~3.2]{Borrvall2003} remark that as $q \to \infty$ the material distribution tends to a 0-1 solution. 

The following existence theorem is due to Borrvall and Petersson \cite[Th.~3.1]{Borrvall2003}.
\begin{theorem}
\label{th:BPexistence}
Suppose that $\Omega \subset \mathbb{R}^d$ is a Lipschitz domain, with $d=2$ or $d=3$ and $\alpha$ {satisfies properties \labelcref{alpha1}--\labelcref{alpha5}. Then, there exists a pair} $(\vect{u}, \rho) \in H^1_{\vect{g},\mathrm{div}}(\Omega) \times C_\gamma$ that minimizes $J$ (as defined in \cref{borrvallmin}). 
\end{theorem}

\section{The deflated barrier method}
\label{sec:dab}
In this section, we summarize how the deflated barrier method finds multiple solutions of fluid topology optimization problems.

\subsection{Forming the barrier functional}
The first step in the algorithm is to form a barrier-like objective from the original objective functional. This is both to aid convergence and to provide an opportunity for deflation to discover multiple local minima (described in Section \ref{sec:deflation}).
Consider the minimization problem: find the minimizers, $(\vect{u}, \rho) \in H^1_{\vect{g},\mathrm{div}}(\Omega)^d \times C_\gamma$, of
\begin{align}
J_\mu^{\epsilon_{\text{log}}}(\vect{u},\rho)\coloneqq J(\vect{u},\rho) - \mu \int_\Omega \log(\rho + \epsilon_{\text{log}}) + \log(1+ \epsilon_{\text{log}}- \rho)\dx,
\label{barriermin} \tag{BF}
\end{align}
where $\mu \ge 0$ is the barrier parameter and $0 < \epsilon_{\text{log}} \ll 1$. We note that the box constraints imposed by the barrier-like terms are never active as $0 \leq \rho \leq 1$ a.e.~in $\Omega$. A similar result to one shown by Evgrafov \cite[Sec.~4]{Evgrafov2014} guarantees the existence of a minimizer to \cref{barriermin} \cite[Prop.~4.2]{Papadopoulos2021e}. As with the original optimization problem, \cref{borrvallmin}, the minimizer is not necessarily unique.

The deflated barrier method targets the first-order optimality conditions of this problem, which we derive next.
Define the following forms to be used in stating the first-order optimality conditions:
\begin{align}
c(\rho, \eta; \vect{u}, \lambda)&\coloneqq\int_\Omega \left[\frac{1}{2} \alpha'(\rho)|\vect{u}|^2 - \frac{\mu}{\rho + \epsilon_{\text{log}}} + \frac{\mu}{1+ \epsilon_{\text{log}} - \rho} + \lambda \right] \eta \; \dx,\\
a(\vect{u},\vect{v}; \rho)&\coloneqq\int_\Omega \alpha(\rho) \vect{u} \cdot \vect{v} + \nu \nabla \vect{u} : \nabla \vect{v} \; \dx, \quad 
l(\vect{v}) \coloneqq \int_\Omega \vect{f} \cdot \vect{v} \; \dx, \\
b(\vect{u},p)&\coloneqq-\int_\Omega  p \divv{\vect{u}} \; \dx,  \quad
d(\lambda, \zeta; \rho)\coloneqq- \zeta \int_\Omega (\gamma - \rho)  \dx.  
\end{align}
The proof of the following proposition, concerning the first-order optimality conditions of \cref{barriermin}, follows from the result in \cite[Prop.~2.5]{Papadopoulos2021e}.
\begin{proposition}
Fix $\mu \geq 0$, $0 < \epsilon_{\mathrm{log}} \ll 1$ and suppose that $(\vect{u},\rho) \in H^1_{\vect{g},\mathrm{div}}(\Omega)^d \times C_\gamma$ is a strict local or global minimizer of \cref{barriermin}. Then, there exist unique Lagrange multipliers $p \in L^2_0(\Omega)$ and $\lambda \in \mathbb{R}$, such that, for all $(\eta, \vect{v}, q, \zeta) \in C_{[0,1]} \times H^1_0(\Omega)^d \times L^2_0(\Omega) \times \mathbb{R}$, the following necessary first-order optimality conditions are satisfied:
\begin{align}
c(\rho, \eta-\rho; \vect{u}, \lambda) &\geq 0, \label{foc:lag1}\\
a(\vect{u},\vect{v}; \rho) + b(\vect{v}, p)&=l(\vect{v}),\label{foc:lag2} \\
b(\vect{u},q)& = 0, \label{foc:lag3} \\
d(\lambda, \zeta;\rho)&=0.  \label{foc:lag4} 
\end{align}
Here $C_{[0,1]} \coloneqq \cup_{\gamma \in [0,1]} C_\gamma$, i.e.~we relax the volume constraint on the variation in the material distribution.
\end{proposition}
Under suitable conditions, \cref{foc:lag1}--\cref{foc:lag4} can be reformulated as a complementarity problem. After a suitable discretization, throughout the deflated barrier method, the nonlinear complementarity problems are solved using the Benson--Munson (BM) algorithm \cite{Benson2003}, a Newton-like algorithm that can enforce the true box constraints on $\rho$. A thorough description of this strategy is given in \cite[Sec.~3.1]{Papadopoulos2021a} and \cref{sec:BMsystem}. We emphasize that the barrier-like terms are  \emph{not} used to enforce box constraints, which are handled by BM, and the enlarged box constraints $[-\epsilon_{\text{log}}, 1 + \epsilon_{\text{log}}]$ are never active for any $\mu \geq 0$. The barrier-like terms are included solely to aid nonlinear convergence via continuation, and to find multiple solutions via deflation.

\subsection{Deflation}
\label{sec:deflation}
Deflation is a mechanism to systemically discover multiple solutions of a nonlinear system with a Newton-like algorithm \cite{Farrell2015, Farrell2019}. Let $Z$ and $Y$ be Banach spaces. Consider the nonlinear system $F(z) = 0,$ $F: Z \to Y$ that has the solutions $z_1, \dots, z_l$. Given a suitable initial guess, a Newton-like algorithm might converge to the solution $z_1$. Now, deflation modifies the nonlinear system in order to remove the known solution whilst retaining undiscovered solutions. This is done via a deflation operator $\mathcal{M}(z;z_1) :Y \to Y$ that ensures that $\mathcal{M}(z;z_1)F(z)$ has a root if the original problem $F(z) = 0$ has an unknown solution, and ensures that a Newton-like method applied to the newly deflated system does not converge to  $z = z_1$.
In this work, the following shifted deflation operator \cite{Farrell2015} is used
\begin{align}
\mathcal{M}(\vect{z}; \vect{z}_1) = \left(\frac{1}{\| \rho - \rho_1\|^2_{L^2(\Omega)}} + 1\right) \mathcal{I},
\end{align}
where $\vect{z} = (\rho, \vect{u}, p, \lambda)$ and $\mathcal{I} :Y \to Y$ is the identity operator. Deflation can be very efficiently implemented; the BM update of the deflated system can be expressed as a (nonlinear) scaling of the BM update of the \emph{undeflated} system evaluated at the same iterate \cite[Sec.~3]{Farrell2015}. This implementation detail is crucial for this work as it implies that preconditioning strategies for the undeflated BM systems can be immediately applied to the deflated systems. A discussion of how the nonlinear scaling is cheaply computed is given in \cite[Sec.~3.2]{Papadopoulos2021a}. 

\subsection{Prediction, continuation, and deflation} The deflated barrier method \cite[Alg.~3.1]{Papadopoulos2021a} is an iterative procedure consisting of three steps. The algorithm is initialized with barrier parameter $\mu = \mu_0$ and the idea is to follow branches (a generalization of the central path) of solutions as $\mu \to 0$. At iteration $k$ of the deflated barrier method, the barrier functional subproblem at $\mu = \mu_k$  is constructed. In the \emph{prediction step}, a cheap initial guess is computed for the barrier functional subproblem from the solution on the same branch at $\mu = \mu_{k-1}$. In the \emph{continuation step}, the BM solver is used to find a solution to the first-order optimality conditions \cref{foc:lag1}--\cref{foc:lag4} at $\mu = \mu_k$, using the initial guess computed in the prediction step, whilst deflating away all other known solutions at $\mu = \mu_k$. In the \emph{deflation step}, we search for new branches of solutions at $\mu = \mu_k$, using solutions discovered at $\mu = \mu_{k-1}$ as initial guesses. The deflation step is terminated if a pre-determined number of branches have already been found or the BM solver fails to converge (by reaching a specified number of iterations without converging). The algorithm terminates when the first-order optimality conditions have been solved (for multiple solutions) at a barrier parameter value of zero.

\section{Discretization and linearization}
\label{sec:discretization}
\subsection{Finite element discretization}
The linear systems that arise in the deflated barrier method will be tied to our choice of finite element discretization. In this work, we choose a discontinuous Galerkin (DG) $k$th-order $\mathrm{BDM}_k \times \mathrm{DG}_{k-1}$ Brezzi--Douglas--Marini  finite element discretization for the velocity and the pressure \cite{Brezzi1987,Brezzi1985}. Let $\mathcal{T}_h$ denote a shape-regular triangulation of the domain $\Omega$. We define the spaces $\vect{X}_{\mathrm{BDM}_k}$ and $\vect{X}_{\mathrm{BDM}_k^{\vect{g}}}$ as:
\begin{align}
\vect{X}_{\mathrm{BDM}_k} &\coloneqq \{ \vect{v} \in L^2(\Omega)^d: \, \divv{\vect{v}} \in L^2(\Omega), \vect{v}|_T \in \mathcal{P}_k(T)^d \,\forall\, T \in \mathcal{T}_h  \},\\
\vect{X}_{\mathrm{BDM}_k^{\vect{g}}} &\coloneqq \{ \vect{v} \in \vect{X}_{\mathrm{BDM}_k} :(\vect{v} - \vect{g}) \cdot \vect{n} = 0 \; \text{on} \; \partial \Omega \},
\end{align}
where $\mathcal{P}_k$ denotes the set of polynomials of order $k$. The degrees of freedom for two dimensions are given in \cite[Sec.~2]{Brezzi1985} and for three dimensions in \cite[Sec.~2]{Brezzi1987}. Similarly, $X_{\mathrm{DG}_{k-1}}$ denotes the set of discontinuous piecewise $(k-1)$-th order polynomials: 
\begin{equation}
X_{\mathrm{DG}_{k-1}} \coloneqq \{ v \in L^2(\Omega): v|_T \in \mathcal{P}_{k-1}(T) \,\forall\, T \in \mathcal{T}_h  \}.
\end{equation}

Some finite element methods for fluid flow, such as the Taylor--Hood finite element pair, do not satisfy the incompressibility constraint $\divv{\vect{u}} = 0$ pointwise. Failure to satisfy the incompressibility constraint pointwise has been observed to support instabilities that result in nonphysical solutions \cite{John2017, Linke2019}. In the BDM finite element pair, $\divv{\vect{X}_{\mathrm{BDM}_k}} \subset X_{\mathrm{DG}_{k-1}}$. Therefore, for any solution $\vect{u}_h$ satisfying the incompressibility constraint \cref{foc:lag3}, we have that $\| \divv{\vect{u}_h} \|_{L^2(\Omega)} = 0$. Hence, the solution is pointwise divergence-free.

The finite element space for the material distribution is denoted by $C_{[0,1],h}$ and is defined by
\begin{align}
C_{[0,1],h} \coloneqq \{ \eta_h \in X_{\mathrm{DG}_0} : 0 \leq \eta_h \leq 1\},
\end{align}
where $X_{\mathrm{DG}_0}$ is the set of piecewise constant finite element functions. As the test functions of $\mathrm{DG}_0$ take the value of one in their respective element and are zero elsewhere, in practice we discretize the material distribution with $\mathrm{DG}_0$ functions and allow the Benson--Munson strategy described below to handle the box constraints.

Since the discontinuous velocity space is not $H^1$-conforming, we use an interior penalty to penalize jumps across edges and faces. We now describe the DG discretization of \cref{foc:lag1}--\cref{foc:lag4} as found in \cite[Sec.~7.1]{Gauger2019} and \cite[Sec.~3]{Papadopoulos2021d}. The interior penalty only arises in the discretization of \cref{foc:lag2}.

Let the set $\mathcal{F}_h$ denote the set of all facets of the triangulation $\mathcal{T}_h$ and $h_F$ represent the diameter of each facet $F \in \mathcal{F}_h$. We split the set of facets into the union $\mathcal{F}_h = \mathcal{F}^i_h \cup \mathcal{F}^\partial_h$ where $\mathcal{F}^i_h$ is the subset of interior facets and $\mathcal{F}^\partial_h$ collects all Dirichlet boundary facets $F \subset \partial \Omega$. For every facet $F \in \mathcal{F}_h$, we assign a unit norm vector $\vect{n}_F$, where if $F \in \mathcal{F}^\partial_h$, then $\vect{n}_F$ is the outer unit normal vector $\vect{n}$. If $F \in \mathcal{F}^i_h$, then $F= \overline{\partial K^+} \cap \overline{\partial K^-}$ for two elements $K^-, K^+ \in \mathcal{T}_h$ and $\vect{n}_F$ points in an arbitrary but fixed direction. Let $\vect{\phi} \in (X_{\DG_k})^d$ and $\tens{\Phi} \in (X_{\DG_k})^{d \times d}$ be any piecewise vector- or matrix-valued function, with traces from within the interior of $K^\pm$ denoted by $\vect{\phi}^\pm$ and $\tens{\Phi}^\pm$, respectively. We define the jump $\lsb \cdot \rsb_F$ and the average $\lcb \cdot \rcb_F$ operators across interior facets $F\in \mathcal{F}^i_h$ by
\begin{align}
\lsb \vect{\phi} \rsb_F = \vect{\phi}^+ \otimes \vect{n}_F^+ + \vect{\phi}^- \otimes \vect{n}_F^- \quad \text{and} \quad \lcb \tens{\Phi} \rcb_F = \frac{1}{2}\left(\tens{\Phi}^+ + \tens{\Phi}^-\right).
\end{align}
If $F \in \mathcal{F}^\partial_h$, we set $\lsb \vect{\phi} \rsb_F = \vect{\phi} \otimes \vect{n}_F$ and $\lcb \tens{\Phi} \rcb_F = \tens{\Phi}$. We note that, for any $F \in \mathcal{F}^\partial_h$, $\int_F |\lsb \vect{\phi} \rsb_F|^2 \, \mathrm{d}s = \int_F | \vect{\phi} |^2 \, \mathrm{d}s$. Finally, for a sufficiently large penalization parameter $\sigma > 0$, we define the broken form $a_h(\vect{u}, \vect{v}; \rho)$ by
\begin{align}
\begin{split}
a_h(\vect{u}, \vect{v}; \rho)&\coloneqq \sum_{K\in \mathcal{T}_h} \int_K  \alpha(\rho) \vect{u} \cdot \vect{v}+ \nu \nabla \vect{u} : \nabla \vect{v} \; \dx + \nu \sum_{F \in \mathcal{F}_h} \sigma h_F^{-1} \int_F  \lsb \vect{u}\rsb_F : \lsb \vect{v}\rsb_F \mathrm{d}s\\
& - \nu \sum_{F \in \mathcal{F}_h}\int_F \lcb \nabla \vect{u} \rcb_F : \lsb \vect{v} \rsb_F \mathrm{d}s -\nu \sum_{F \in \mathcal{F}_h} \int_F  \lsb \vect{u} \rsb_F : \lcb \nabla \vect{v} \rcb_F \mathrm{d}s,
\end{split}
\end{align}
and the linear functional $l_h$ as
\begin{align}
l_h(\vect{v})&\coloneqq  \int_\Omega \vect{f} \cdot \vect{v} \; \dx + \nu\sum_{F \in \mathcal{F}^\partial_h} \sigma h_F^{-1} \int_F \lsb \vect{g} \rsb_F : \lsb \vect{v} \rsb_F \; \mathrm{d}s - \nu \sum_{F \in \mathcal{F}^\partial_h} \int_F \lsb \vect{g} \rsb_F : \lcb \nabla \vect{v} \rcb_F \mathrm{d}s.
\end{align}
Then, given a barrier parameter $\mu$ and a penalization parameter $\sigma > 0$, the discretized deflated barrier subproblem is to find $(\rho_h, \vect{u}_h, p_h, \lambda_h) \in C_{\gamma,h} \times \vect{X}_{\mathrm{BDM}^{\vect{g}}_k} \times X_{\mathrm{DG}_{k-1}} \times \mathbb{R}$ such that, for all $(\eta_h, \vect{v}_h, q_h, \zeta_h) \in C_{[0,1],h} \times \vect{X}_{\mathrm{BDM}^{\vect{0}}_k} \times (X_{\mathrm{DG}_{k-1}} \backslash \mathbb{R}) \times \mathbb{R}$, we have:
\begin{align}
c(\rho_h, \eta_h - \rho_h; \vect{u}_h, \lambda_h) &\geq 0, \label{dfoc1}\\
a_h(\vect{u}_h, \vect{v}_h;\rho_h) +  b(\vect{v}_h,p_h)&=l_h(\vect{v}_h),\label{dfoc2}\\
b(\vect{u}_h,q_h) &=0, \label{dfoc3}\\
d(\lambda_h, \zeta_h; \rho_h) &=0.\label{dfoc4}
\end{align}
 Let the broken $H^1$-norm $\| \cdot \|^2_{H^1_{\vect{g}}(\mathcal{T}_h)}$ be defined as
\begin{align}
\begin{split}
\| \vect{v} \|^2_{H^1_{\vect{g}}(\mathcal{T}_h)} &\coloneqq \|\vect{v}\|^2_{L^2(\Omega)} + \sum_{K \in \mathcal{T}_h} \| \nabla \vect{v} \|^2_{L^2(K)} \\
&\indent + \sum_{F \in \mathcal{F}^i_h} \int_F h_F^{-1} | \lsb \vect{v} \rsb_F|^2 \mathrm{d}s+ \sum_{F \in \mathcal{F}^\partial_h} \int_F h_F^{-1} | \lsb \vect{v} - \vect{g} \rsb_F|^2 \mathrm{d}s. 
\end{split}
\end{align}
Building on previous work \cite{Papadopoulos2022a}, it was shown by Papadopoulos \cite{Papadopoulos2021d} that, for every isolated minimizer $(\vect{u}, \rho, p)$ of \cref{borrvallmin}, there exists a sequence of discretized solutions $(\vect{u}_h, \rho_h, p_h, \lambda_h)$ to \cref{dfoc1}--\cref{dfoc4}, such that, as $h \to 0$, $\|\vect{u} - \vect{u}_h\|_{H^1_{\vect{g}}(\mathcal{T}_h)} \to 0$, $\|\rho - \rho_h\|_{L^s(\Omega)} \to 0$, $s \in [1,\infty)$, and $\|p - p_h\|_{L^2(\Omega)} \to 0$. By extrapolating known results of a BDM discretization for the Stokes and Stokes--Brinkman equations \cite{Brezzi1987, Brezzi1985, Konno2012} we expect a first-order $\mathrm{BDM}_1$ discretization to converge at a rate of $\mathcal{O}(h)$ in the broken $H^1$-norm for the velocity, the $L^2$-norm for the pressure, and the $L^2$-norm for the material distribution. Numerical evidence for these rates of convergence is given in \cite[Fig.~2]{Papadopoulos2021d}.

\begin{remark}
The value of $\alpha(\rho)$ in the momentum equation \cref{foc:lag2} can range between 0 and $\bar \alpha$ where in practice $\bar \alpha \sim \mathcal{O}(10^4)$. Hence, an element that is robust to the transition between Stokes and Darcy flow may have better accuracy, e.g.~the Mardal--Tai--Winther (MTW) finite element \cite{Mardal2002}. The MTW finite element also produces pointwise divergence-free solutions and the characterization of the kernel of the divergence operator is known \cite[Sec.~4.2]{Mardal2002}. In particular, it has been shown that vertex-star patch relaxation is also effective for multigrid cycles involving MTW finite elements \cite[Sec.~7.2]{Aznaran2021}. Hence, we believe the preconditioner described below would still be effective. However, unlike BDM finite elements with an interior penalty, there are currently no convergence results for an MTW finite element discretization to minimizers of the Borrvall--Petersson problem.
\end{remark}

\subsection{The Benson--Munson linear system}
\label{sec:BMsystem}
The BM solver \cite{Benson2003} attempts to find a solution of a complementarity problem via linearizations of the residual constrained to the inactive set. First, the discrete Newton system is formed and the active set is defined. The active set contains the degrees of freedom that satisfy a strict complementarity condition in the primal and residual vectors. Next, the rows and columns of the Jacobian in the Newton system associated with the active set degrees of freedom are set to those of the identity. Finally, the rows on the right-hand side vector associated with the active set degrees of freedom are fixed to zero. Once the update, $\delta \vectt{z}$, of this modified system is computed, the new iterate $\vectt{z}^{k+1} = \vectt{z}^{k} + \delta \vectt{z}$ is component-wise projected onto the box constraints.

Let $\vectt{f}: \mathbb{R}^n \to \mathbb{R}^n$ and consider the box constraints $\vectt{a}, \vectt{b} \in \mathbb{R}^n$ where $\vectt{a}_i < \vectt{b}_i$ for all $i= 1,\dots,n$. Consider the mixed complementarity problem given by 
\begin{alignat}{2}
\text{either} &\;\; \vectt{a}_i < \vectt{z}_i < \vectt{b}_i \;\; &&\text{and} \;\; \vectt{f}(\vectt{z})_i = 0, \label{eq:mcp1}\\
\text{or} &\;\; \vectt{a}_i = \vectt{z}_i \;\; &&\text{and} \;\; \vectt{f}(\vectt{z})_i \geq 0,\\
\text{or} &\;\; \vectt{z}_i = \vectt{b}_i \;\; &&\text{and} \;\; \vectt{f}(\vectt{z})_i \leq 0. \label{eq:mcp3}
\end{alignat}
We define $\hat{\vectt{f}}: \{ \vectt{z} \in \mathbb{R}^n : \vectt{a}_i \leq \vectt{z}_i \leq \vectt{b}_i, \; i=1,\dots,n\} \to \mathbb{R}^n$ as follows, for $i=1,\dots,n$,
\begin{align}
[\hat{\vectt{f}}(\vectt{z})]_i \coloneqq
\begin{cases}
\vectt{f}(\vectt{z})_i & \text{if} \;\; \vectt{a}_i < \vectt{z}_i < \vectt{b}_i,\\
\min\{\vectt{f}(\vectt{z})_i, 0\} & \text{if}\;\; \vectt{z}_i = \vectt{a}_i,\\
\max\{\vectt{f}(\vectt{z})_i, 0\}& \text{if}\;\; \vectt{z}_i = \vectt{b}_i.
\end{cases}
\end{align}
We denote the component-wise projection onto the true box constraints by $\pi$, i.e.
\begin{align}
\pi(\vectt{z})_i = 
\begin{cases}
\vectt{a}_i & \text{if} \; \vectt{z}_i < \vectt{a}_i,\\
\vectt{z}_i&\text{if} \; \vectt{a}_i \leq \vectt{z}_i \leq \vectt{b}_i,\\
\vectt{b}_i & \text{if} \; \vectt{z}_i> \vectt{b}_i.
\end{cases}
\end{align}

Consider a matrix $\tens{A} \in \mathbb{R}^{n \times n}$ and the subset of indices $S \subset \{1,\dots, n\}$. Then, the matrix $\tens{A}_{S, S} \in \mathbb{R}^{n \times n}$ is defined by zeroing the rows and columns of $\tens{A}$ in $\{1,\dots,n\}\backslash S$ and placing a 1 on the diagonal of the zeroed rows and columns. Similarly for any column vector $\vectt{x} \in \mathbb{R}^n$, the vector $\vectt{x}_{S} \in \mathbb{R}^{n}$ is constructed by zeroing the rows in the set  $\{1,\dots,n\}\backslash S$ of $\vectt{x}$. 

\begin{remark}
In the original BM solver found in \cite[Sec.~3.2]{Benson2003}, the rows and columns are eliminated rather than zeroed, resulting in a reduced matrix. However, the BM solver as described in \cref{alg:bm} and \cite[Sec.~3.2]{Benson2003} are equivalent.
\end{remark}

\begin{algorithm}[ht]
\caption{Benson--Munson solver \cite[Sec.~3.2]{Benson2003}}
\label{alg:bm}
\begin{algorithmic}[1]
\Initialize{
$k \gets 0$ \Comment{Initial iteration number}\\
$\vectt{a}, \vectt{b}$ \Comment{Box constraints}\\
$\vectt{z}^0$ \Comment{Feasible initial guess}\\ 
tol \Comment{Approximate solve tolerance}\\
\vspace{2mm}
}
\While{$\| \hat{\vectt{f}}(\vectt{z}^k)\|_2 > \mathrm{tol}$}
\State{Define the active set $\mathcal{A}^k$ and the inactive set $\mathcal{I}^k$ as:
\begin{align*}
\mathcal{A}^k &\coloneqq \{i: \vectt{z}^k_i=\vectt{a}_i \; \text{and} \; \vf(\vectt{z}^k)_i > 0 \} \cup \{i: \vectt{z}^k_i=\vectt{b}_i \; \text{and} \; \vf(\vectt{z}^k)_i <  0 \},\\
\mathcal{I}^k &\coloneqq  \{i\}_{i=1}^n \backslash \mathcal{A}^k.
\end{align*}}
\State{Solve the BM linear system:
$\vectt{f}'(\vectt{z}^k)_{\mathcal{I}^k,\mathcal{I}^k}\delta \vectt{z} = -\vectt{f}(\vectt{z}^k)_{\mathcal{I}^k}$.}
\State{$\vectt{z}^{(k+1)} \gets \pi(\vectt{z}^k + \beta^k \delta \vectt{z})$}, where $\beta^k$ is a (possibly adaptive) linesearch. 
\State{$k \gets k+1$}
\EndWhile
\end{algorithmic}
\end{algorithm}

When applied to a linear elliptic problem, and if the active and inactive sets of the two algorithms coincide, the updates computed by BM and Hinterm\"uller et al.'s primal-dual active set strategy \cite{HintermullerIto2003} are identical \cite[Th.~A.1]{Papadopoulos2021a}. Under suitable conditions, the primal-dual active set strategy can be shown to be a semismooth Newton method \cite{Qi1993, Qi1993b, Ulbrich2003}. Although the problems we consider here do not satisfy the assumptions made in Theorem A.1 of \cite{Papadopoulos2021a}, we do observe local superlinear convergence in practice.

We now derive the linear systems that arise when using the BM active set strategy to solve the nonlinear system \cref{dfoc1}--\cref{dfoc4}. Denote the basis functions of the finite element spaces of the material distribution, the velocity, the pressure, and $\mathbb{R}$ by $\eta_i, \vect{\phi}_i, \psi_i$, and $r$, respectively and the number of degrees of freedom by $n_\rho, n_{\vect{u}}$, $n_p$, and 1, respectively, so that the total number of degrees of freedom of the system is given by $n = n_\rho + n_{\vect{u}} + n_p + 1$. Consider the finite element BM iterate $\vect{z}^k_{h} = (\rho^k_{h}, \vect{u}^k_{h}, p^k_{h}, \lambda^k_{h})$. Let $\vf^k: \mathbb{R}^n \to \mathbb{R}^n$ denote the nonlinear residual induced by the complementarity reformulation of \cref{dfoc1}--\cref{dfoc4}. In the following, we drop the superscript iteration number $^{k}$ for clarity. Let $\vectt{z}$ denote the discrete coefficient vector of $\vect{z}_h$. In the context of \cref{alg:bm}, the box constraints take the values $\vectt{a}_i = 0$ and $\vectt{b}_i=1$ for all degrees of freedom associated with $\rho_h$ and $\vectt{a}_i = -\infty$, $\vectt{b}_i=+\infty$, otherwise. The BM linear system, as solved on line 4 of \cref{alg:bm}, is the following:
\begin{align}
\tH_{\rho,\vect{u},p,\lambda} \delta \vectt{z} = \begin{pmatrix}
\tC_\mu & \tD^\top                        & \tzero            & \tE^\top \\
\tD        & \tA & \tB^\top & \tzero             \\
\tzero      & \tB                                & \tzero            & \tzero            \\
\tE        & \tzero                                   & \tzero            & \tzero             
\end{pmatrix}
\begin{pmatrix}
\delta \vectt{\rho} \\
\delta \vectt{u} \\
\delta \vectt{p} \\
\delta \vectt{\lambda}
\end{pmatrix}
=
-\begin{pmatrix}
\vectt{f}_{\rho} \\
\vectt{f}_{\vect{u}} \\
\vectt{f}_{p} \\
\vectt{f}_{\lambda}
\end{pmatrix}
= -\vectt{f},
\label{eq:DiscretizedNewton}
\end{align}
where $\delta \vectt{\rho}, \delta \vectt{u}, \delta \vectt{p}$ and $\delta \vectt{\lambda}$ denote the discrete coefficient vector BM updates for $\rho, \vect{u},  p$ and $\lambda$, and $\vectt{f}_{\rho}$, $\vectt{f}_{\vect{u}}$, $\vectt{f}_{p}$, and $\vectt{f}_{\lambda}$ are the corresponding blocks of the nonlinear residual with the active set rows, $i \in \mathcal{A}$, in $\vectt{f}_{\rho}$ zeroed. The entries of $\tA$ and $\tB$ are given by
\begin{align}
[\tA]_{ij} =  a_{h}( \vect{\phi}_j ,  \vect{\phi}_i; \rho_h) \text{ and } [\tB]_{ij} = b(\vect{\phi}_j,\psi_i).
\end{align}
Furthermore, if $j \in \mathcal{I}$, then
\begin{align}
[\tD]_{ij} = \int_\Omega (\alpha'(\rho_h) \vect{u}_h \cdot \vect{\phi}_i) \eta_j \; \dx, \;\;\;
[\tE]_{ij} = - r \int_\Omega \eta_j \; \dx.
\end{align}
Otherwise if $j \in \mathcal{A}$, then $[\tD]_{ij} = 0$ and $[\tE]_{ij} = 0$ for all $i$. Finally if $i, j \in \mathcal{I}$ then
\begin{align}
[\tC_\mu]_{ij}= \int_\Omega \left[\frac{1}{2} \alpha''(\rho_h)|\vect{u}_h|^2 + \frac{\mu}{(\rho_h + \epsilon_{\text{log}})^2} + \frac{\mu}{(1+ \epsilon_{\text{log}} - \rho_h)^2} \right] \eta_i \eta_j \dx. \label{Cmu}
\end{align}
Otherwise, if $i\in \mathcal{A}$ or $j \in \mathcal{A}$, then $[\tC_\mu]_{ij} = \delta_{ij}$, where $\delta_{ij}$ is the Kronecker delta. 
\begin{remark}
$\tE$ is a row vector of size $1 \times n_\rho$. 
\end{remark}
 
In the remainder of this subsection we discuss the invertibility of the matrix in \cref{eq:DiscretizedNewton} and its subblocks.

\begin{proposition}
\label{prop:posdef}
The matrix $\tA\in \mathbb{R}^{n_{\vect{u}} \times n_{\vect{u}}}$ is symmetric, and provided the penalization parameter $\sigma>0$ is sufficiently large, then it also positive-definite.
\end{proposition}
\begin{proof}
Symmetry is realized by swapping the indices $i$ and $j$ of the basis functions in their respective definitions and noting that the resulting integrals are equal. Positive-definiteness of $\tA$, for sufficiently large $\sigma >0$, follows from $\alpha(\rho) \geq 0$ and \cite[Sec.~3.3]{Hong2016}.
\end{proof}

\begin{proposition}
\label{prop:CmuActiveSet}
The matrix $\tC_\mu$ is symmetric positive semi-definite. Moreover, if either $\mu > 0$ or $|\vect{u}_h| > 0$ a.e., then $\tC_\mu$ is symmetric positive-definite.
\end{proposition}
\begin{proof}
The symmetry of $\tC_\mu$ is realized by swapping the indices $i$ and $j$ of the basis functions in its definition and noting that the resulting integrals are equal. Consider the unmodified matrix $\hat{\tC}_\mu$, defined by \cref{Cmu} for all $i$ and $j$. Pick an arbitrary function $\eta_h \in X_{\DG_0}$ with discrete coefficient vector $\vectt{\eta} \in \mathbb{R}^{n_\rho}$. We note that
\begin{equation}
\vectt{\eta}^\top \hat{\tC}_\mu \vectt{\eta} =  \int_\Omega \left[\frac{1}{2} \alpha''(\rho_h)|\vect{u}_h|^2 + \frac{\mu}{(\rho_h + \epsilon_{\text{log}})^2} + \frac{\mu}{(1+ \epsilon_{\text{log}} - \rho_h)^2} \right] |\eta_h|^2 \dx. \label{eq:posdefproof1}
\end{equation}
All the terms in the integral \cref{eq:posdefproof1} are non-negative. Hence, $\hat{\tC}_\mu$ is positive semi-definite. 

Assumption \labelcref{alpha5} implies that $\alpha''(\rho_h) > 0$. If $\mu > 0$, then the rational expressions are strictly greater than zero as $0 \leq \rho_h \leq 1$. Otherwise, if $|\vect{u}_h| > 0$ a.e., then the term $\alpha''(\rho_h) |\vect{u}_h|^2 > 0$ a.e. Hence, if either $\mu>0$ or $|\vect{u}_h| > 0$ a.e.~the right-hand side of \cref{eq:posdefproof1} is equal to zero if and only if $\eta_h = 0$, which is true if and only if $\vectt{\eta} = \vectt{0}$. Therefore, if either $\mu>0$ or $|\vect{u}_h| > 0$ a.e., $\vectt{\eta}^\top \hat{\tC}_\mu \vectt{\eta} \geq 0$ with equality if and only if $\vectt{\eta} = \vectt{0}$. Hence, $\hat{\tC}_\mu$ is symmetric positive-definite. 

Since the discretization for $\rho$ is piecewise constant, $\hat{\tC}_\mu$ is a diagonal matrix and, therefore, all diagonal entries must be positive. The procedure of zeroing rows and columns associated with the BM active set and replacing the diagonal entry with a one will result in a diagonal matrix with non-negative (positive if either $\mu>0$ or $|\vect{u}_h| > 0$ a.e.) diagonal entries. We conclude that $\tC_\mu$ must be symmetric positive semi-definite and if either $\mu>0$ or $\vect{u}_h > 0$ a.e., then it is symmetric positive-definite.
\end{proof}

\begin{remark}
\label{rem:B-rank}
For problems with a pure Dirichlet boundary condition, $\tH_{\rho,\vect{u},p,\lambda}$ will have a nullspace of at least one dimension associated with the fact that the pressure is only unique up to a constant \cite[Ch.~4]{Elman2014}. Since we will use an FGMRES Krylov method on the outermost level of the linear solve, this particular nullspace can be handled by FGMRES without the method breaking down \cite[Ch.~9.3.5]{Elman2014}. To show that the nullspace is indeed the space of constants with respect to the pressure, we assume that $\tens{B}$ has been modified so that it is of full row rank, e.g.~by picking a pressure degree of freedom and fixing it to zero. 
\end{remark}

\begin{proposition}
Suppose that either $\mu > 0$ or $|\vect{u}_h| > 0$ a.e. Moreover, assume that $\tA - \tD \tC_\mu^{-1} \tD^\top$ is symmetric positive-definite and $\tB$ has been modified so that it has full row rank. Then, the matrix  $\tH_{\rho,\vect{u},p,\lambda}$, as defined in \cref{eq:DiscretizedNewton}, is invertible.
\end{proposition}

\begin{proof}
If either $\mu > 0$ or $|\vect{u}_h| > 0$ a.e.~then $\tC_\mu$ is symmetric positive-definite by \cref{prop:CmuActiveSet} and, therefore, invertible. Thus the matrix $\tA - \tD \tC_\mu^{-1} \tD^\top$ is well-defined. By assumption $\tA - \tD \tC_\mu^{-1} \tD^\top$ is symmetric positive-definite. Therefore, since  $\tA$ is symmetric positive-definite by \cref{prop:posdef}, we have that the following matrix is also symmetric positive-definite \cite[Th.~1.12]{Zhang2006}:
\begin{align}
\tens{G} \coloneqq
\begin{pmatrix}
\tC_\mu & \tD^\top \\
\tD        & \tA
\end{pmatrix}.
\end{align}
We define $\tens{B}_0 \coloneqq \begin{pmatrix}\tens{0} & \tens{B} \end{pmatrix} \in \mathbb{R}^{n_p \times(n_\rho  + n_{\vect{u}})}$ and $\tens{E}_0 \coloneqq \begin{pmatrix}\tens{E} & \tens{0} \end{pmatrix}\in \mathbb{R}^{1 \times (n_\rho  + n_{\vect{u}})}$  and re-block $\tH_{\rho,\vect{u},p,\lambda}$ as follows:
\begin{align}
\tH_{\rho,\vect{u},p,\lambda} = \begin{pmatrix}
\tens{G} & \tens{B}_0^\top & \tens{E}^\top_0 \\
\tens{B}_0 & \tens{0} & \tens{0} \\
\tens{E}_0 & \tens{0} & \tens{0}
\end{pmatrix}.
\label{eq:re-block}
\end{align}
The matrix on the right-hand side of \cref{eq:re-block} is of double-saddle point type. $\tE_0$ is a nonzero row vector and thus must have full row rank. Moreover, by assumption $\tB$ has been modified so that it has full row rank. Hence, $\tB_0$ also has full row rank. Consider any vector $\vectt{x} = (\vectt{x}_1^\top \;\;\; \vectt{x}_2^\top)^\top \in \mathbb{R}^{n_\rho  + n_{\vect{u}}}$. We note that $\tE_0^\top \vectt{x}^\top = \tE^\top \vectt{x}_1^\top$ and $\tB_0^\top \vectt{x}^\top = \tB^\top \vectt{x}_2^\top$. Hence, $\mathrm{range}(\tens{B}_0^\top) \cap \mathrm{range}(\tens{E}_0^\top) = \{\vect{0}\}$. Thus $\tH_{\rho,\vect{u},p,\lambda}$ is invertible \cite[Prop.~2.3]{Ali2018}. 
\end{proof}

\begin{remark}
\label{rem:Schur-pos-def}
The proof of the invertibility of $\tH_{\rho,\vect{u},p,\lambda}$ (provided the pressure null space has been eliminated) relied heavily on the assumption that $\tA - \tD \tC_\mu^{-1} \tD^\top$ is symmetric positive-definite. Although the symmetry is guaranteed, positive-definiteness is not. Numerically, we found that $\tA - \tD \tC_\mu^{-1} \tD^\top$ is often positive-definite for large ranges of $\mu$, choices of $\alpha$, and BM iterates $\rho_{h,k}$ and $\vect{u}_{h,k}$. Despite this, there are choices where $\tA - \tD \tC_\mu^{-1} \tD^\top$ has negative eigenvalues. In practice we have not encountered a case where $\tA - \tD \tC_\mu^{-1} \tD^\top$ is singular. 
\end{remark}

\begin{remark}
\label{rem:Schur-indefinite}
A full proof of the invertibility of $\tH_{\rho,\vect{u},p,\lambda}$ in the case where $\tA - \tD \tC_\mu^{-1} \tD^\top$ is indefinite or singular is beyond the scope of this work. However, we note the following. The first-order optimality conditions \cref{foc:lag1}--\cref{foc:lag4} can be rewritten as a semismooth system of equations. Suppose, after discretization, we denote this system $F :\mathbb{R}^{n_{\vect{u}} + 3 n_\rho + n_p + 1} \to \mathbb{R}^{n_{\vect{u}} + 3 n_\rho + n_p + 1}$ where the extra $2 n_\rho$ degrees of freedom are associated with the Lagrange multipliers that enforce the box constraints on $\rho$. Then, provided one can show that $F$ is locally Lipschitz continuous, $F$ is semismooth and $F'(\vectt{z}_*)$ is invertible at $\vectt{z}_*$ where $\vectt{z}_*$ satisfies $F(\vectt{z}_*) = 0$, then one may be able to invoke a semismooth analogue of the Rall--Rheinboldt theory \cite{Farrell2019}. Hence, provided the iterate $\vectt{z}_k$ is sufficiently close to $\vectt{z}_*$, then $F'(\vectt{z}_k)$ is invertible. However, we note that although the BM solver has been shown to be almost equivalent to a semismooth Newton method for linear elliptic problems in \cite[App.~A]{Papadopoulos2021a}, that theory does not extend to the systems we are solving here. Hence the invertibility of $F'(\vectt{z}_k)$ would not necessarily imply the invertibility of $\tH_{\rho,\vect{u},p,\lambda}$.
\end{remark}

\section{Preconditioning}
\label{sec:preconditioning}
In this section, we develop a preconditioner for solving \cref{eq:DiscretizedNewton}. As discussed in \cref{sec:deflation}, preconditioning strategies that are robust for the undeflated system can also be used to compute solutions of the deflated systems.  On the outermost level of the deflated barrier method, we perform continuation in the barrier parameter $\mu$. Next, at a given $\mu$, we use the BM solver to find a solution of \cref{dfoc1}--\cref{dfoc4}.
A direct sparse LU factorization of the matrix in \cref{eq:DiscretizedNewton} is infeasible on fine meshes of three-dimensional problems. Thus we turn to preconditioning techniques to reduce the cost of each inner linear solve. The preconditioning is made difficult by the saddle point nature of the matrix in \cref{eq:DiscretizedNewton} and the barrier-like terms in $\tC_\mu$. In the following subsections we introduce a nested block preconditioning method for solving \cref{eq:DiscretizedNewton}, where the Schur complements are controlled with an augmented Lagrangian term.
As outermost Krylov solver, we use a preconditioned FGMRES method \cite{Saad1993}.

\subsection{Block preconditioning}
\label{sec:explore}
Consider the well-posed linear system 
\begin{align}
\begin{pmatrix}
\mathbb{A} & \mathbb{B}\\
\mathbb{C} & \mathbb{D}
\end{pmatrix}
\begin{pmatrix}
\vectt{x} \\
\vectt{y}
\end{pmatrix}
=
\begin{pmatrix}
\vectt{c} \\
\vectt{d}
\end{pmatrix},
\label{eq:linearsystem}
\end{align}
where $\mathbb{A} \in \mathbb{R}^{n_1 \times n_1}$ is invertible, $\mathbb{B} \in \mathbb{R}^{n_1 \times n_2}$, $\mathbb{C} \in \mathbb{R}^{n_2 \times n_1}$ and $\mathbb{D} \in \mathbb{R}^{n_2 \times n_2}$. Then, under suitable conditions~\cite[\S 3.2]{Benzi2005} the inverse of the matrix in \cref{eq:linearsystem} admits a full block factorization of the form
\begin{align}
\begin{pmatrix}
\mathbb{A} & \mathbb{B}\\
\mathbb{C} & \mathbb{D}
\end{pmatrix}^{-1}
= 
\begin{pmatrix}
I  & -\mathbb{A}^{-1} \mathbb{B}\\
0& I 
\end{pmatrix}
\begin{pmatrix}
\mathbb{A}^{-1}  &0 \\
0& \mathbb{S}^{-1} 
\end{pmatrix}
\begin{pmatrix}
I  &  0\\
-\mathbb{C} \mathbb{A}^{-1} & I 
\end{pmatrix},
\label{eq:blockprecon}
\end{align}
where $\mathbb{S} = \mathbb{D} - \mathbb{C} \mathbb{A}^{-1} \mathbb{B}$. Preconditioners for \cref{eq:linearsystem} can be found by developing cheap approximations to $\mathbb{A}^{-1}$ and $\mathbb{S}^{-1}$ and substituting them into \cref{eq:blockprecon} \cite{Murphy2000, Wathen2015}. 

The subspace spanned by the volume constraint Lagrange multiplier $\lambda$ is one-dimensional and can be handled by at most one iteration of a Krylov subspace solver or via block preconditioning. Experimentally, we found that a full block preconditioner  of the real block performed best. Writing the density-momentum-pressure block,
\begin{align}
\tH_{\rho, \vect{u}, p} \coloneqq 
\begin{pmatrix}
\tC_\mu & \tD^\top                        & \tzero          \\
\tD        & \tA  & \tB^\top       \\
\tzero         & \tB                                & \tzero                
\end{pmatrix},
\label{rho-u-p-block}
\end{align}
we choose $\mathbb{A} = \tH_{\rho, \vect{u}, p}$, $\mathbb{B} = \tE^\top$, $\mathbb{C} = \tE$, and $\mathbb{D} = \tzero$. The Schur complement, 
\begin{align}
\mathbb{S} = \tS_0\coloneqq -\tE \tH_{\rho, \vect{u}, p}^{-1} \tE^\top, \label{eq:lambdaSchurComplement}
\end{align}
is a $1 \times 1$ matrix and can be inverted by taking its reciprocal. Hence, the difficulty now lies in solving linear systems involving \cref{rho-u-p-block}. Since we have only decreased the size of the linear system by one dimension, an LU factorization is still infeasible and we consider block preconditioners for \cref{rho-u-p-block}. We summarize the initial components of the solver in \cref{fig:flowbeginning}.
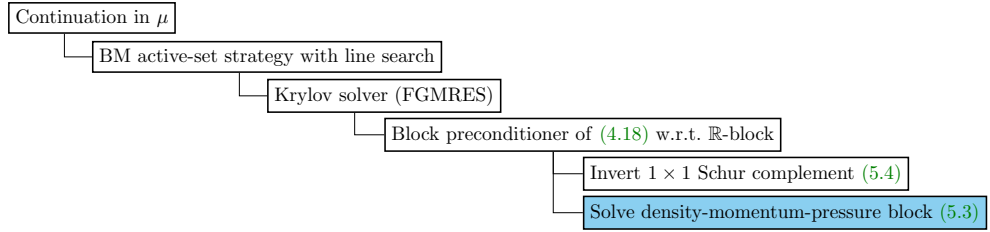
\begin{figure}[ht]
\centering
 \begin{tikzpicture}[%
every node/.style={draw=black, thick, anchor=west},
grow via three points={one child at (0.0,-0.7) and
    two children at (0.0,-0.7) and (0.0,-1.4)},
edge from parent path={(\tikzparentnode.210) |- (\tikzchildnode.west)}]
\node {Continuation in $\mu$}
child { node {BM active-set strategy with line search}
    child { node {Krylov solver (FGMRES)}
        child { node{Block preconditioner of \cref{eq:DiscretizedNewton} w.r.t.~$\mathbb{R}$-block}
            child { node {Invert $1 \times 1$ Schur complement \cref{eq:lambdaSchurComplement}}}
            child {node[fill=babyblue]{Solve density-momentum-pressure block \cref{rho-u-p-block}}
            }          
        }
    }
};
\end{tikzpicture}
\caption{Main components of the deflated barrier method solver. The remainder of this section focuses on developing preconditioners for the item in blue.}\label{fig:flowbeginning}
\end{figure}

We precondition \eqref{rho-u-p-block} by taking the Schur complement with respect to the momentum-pressure block. This approach was utilized by Evgrafov for preconditioning the linear systems arising in a similar solver \cite[Sec.~5]{Evgrafov2014}. In the notation of \cref{eq:linearsystem}, $\mathbb{A} = \tC_\mu$, $\mathbb{B} = (\tD^\top \;\; \tzero)$, $\mathbb{C} = (\tD \;\; \tzero)^\top$ and,
\begin{align}
\mathbb{D} = \tens{Q} \coloneqq
\begin{pmatrix}
\tA   & \tB^\top       \\
\tB   & \tzero                
\end{pmatrix}.
\label{eq:approxSchur1}
\end{align}
We know that $\tC_\mu$ is invertible by \cref{prop:CmuActiveSet}. Hence, we write $\mathbb{A}^{-1} = \tC_\mu^{-1}$.
The matrix $\tens{Q}$ resembles the linear system that arises in the discretization of the Stokes--Brinkman equations.
The Schur complement is given by
\begin{align}
\mathbb{S} = \tS_1 \coloneqq 
\tens{Q}
- 
\begin{pmatrix}
\tD \\
\tzero               
\end{pmatrix}
\tC_\mu^{-1}
\begin{pmatrix}
\tD ^\top & \tzero                
\end{pmatrix}
= 
\begin{pmatrix}
\tA - \tD \tC_\mu^{-1} \tD^\top  & \tB^\top       \\
\tB                                & \tzero               
\end{pmatrix}.
\label{eq:Schur1}
\end{align}
The reason we use a  $\mathrm{DG}_0$ piecewise constant discretization for the material distribution is to ensure that $\tS_1$ is sparse. Since the material distribution is discretized with $\mathrm{DG}_0$ finite elements, $\tC_\mu$ is a diagonal matrix, $\tD \tC_\mu^{-1} \tD^\top$ is still sparse and hence $\tS_1$ is also sparse. 

Thus far, the application of block preconditioning has reduced the solution of the full matrix \cref{eq:DiscretizedNewton} to the following:
\begin{enumerate}
\itemsep=0pt
\item Apply an outer FGMRES solver;
\item Apply the reciprocal of $\tS_0 \in \mathbb{R}$;
\item Invert the diagonal matrix $\tC_\mu$;
\item Apply the inverse of the $2 \times 2$ block matrix $\tS_1$, which is the same size as the matrix that arises in a discretized pure Stokes problem. 
\end{enumerate}
For now we assume that $\tS_1$ is invertible and defer discussions of invertibility to later.

We must now develop solvers for $\tS_1$ as given in \cref{eq:Schur1}. One option is to use a direct solver. However, we can further reduce the computational work with another application of block preconditioning. Consider taking the inner Schur complement in $\tS_1$ with respect to the pressure block. Using the notation of \cref{eq:linearsystem}, $\mathbb{B} = \tB^\top$, $\mathbb{C} = \tB$, and $\mathbb{A} = \tA - \tD \tC_\mu^{-1} \tD^\top$. The innermost Schur complement takes the form
\begin{align}
\mathbb{S} =\tS_2 \coloneqq -\tB(\tA - \tD \tC_\mu^{-1} \tD^\top)^{-1}\tB^\top.
\label{eq:schur2}
\end{align}
This time, $\tS_2$ is dense, and we employ an augmented Lagrangian approach.

We first recall some known results for the Stokes and Stokes--Brinkman equations. Let $-\Delta_h$ denote the negative discretized Laplacian matrix. In the context of the incompressible Stokes equations, $\tB(-\Delta_h)^{-1}\tB^\top$ is spectrally equivalent to the viscosity-scaled pressure mass matrix $\nu^{-1}\tM_p$ \cite{Silvester1994, Wathen1993}, see also \cite[Th.~5.22]{Elman2014}. $\tM_p$ is a sparse mass matrix and can be cheaply factorized or solved with a multigrid method. Therefore, a good approximation to the Schur complement of the pure Newtonian Stokes problem is given by $\nu^{-1} \tM_p$. The idea is that the momentum block can then be solved with a direct solver, multigrid methods, or other alternative solvers. Unfortunately, in the context of the Stokes--Brinkman equations, Popov \cite{Popov2012} noted that the presence of the Brinkman term $\alpha(\rho_h)\vect{u}_h$ in the momentum block $\tA$ renders the approximation given by $\nu^{-1} \tM_p$ ineffective. Popov proposed a Schur complement preconditioning technique based on incomplete LU factorization, but such factorizations do not generally yield mesh-independent preconditioners. An alternative is a preconditioning scheme utilized by Borrvall and Petersson in their original paper \cite[Sec.~2.6]{Borrvall2003} based on the work of Cahouet and Chabard \cite{Cahouet1988}. However, during numerical experiments, we found that an augmented Lagrangian approach performed better for the problems in this work, see \cite[Ch.~5.3.1]{Papadopoulos2021e}.

We now propose an augmented Lagrangian strategy to control the second Schur complement $\tS_2$. The essential idea is to add a term to the momentum equation that does not change the exact solution, but does change the Schur complement $\tS_2$; in particular, by scaling the augmentation appropriately, $\tS_2$ can be almost perfectly approximated with a scaled pressure mass matrix. The cost of this approach is that it makes the augmented momentum block more difficult to solve. The augmented Lagrangian approach has been shown to be robust for a variety of difficult saddle-point systems such as the stationary Navier--Stokes equations at high Reynolds number \cite{Farrell2019a}, implicitly-constituted anisothermal non-Newtonian flow \cite{Farrell2020d}, and magnetohydrodynamics \cite{Laakmann2021}. Hence, this approach has potential for extension to different fluid topology optimization problems.

There are two possible augmented Lagrangian approaches: continuous and discrete. These approaches are mathematically equivalent for exactly divergence-free elements such as the one employed here. We choose to introduce the method in the discrete setting. Post-discretization, the matrix $\tA$ in \cref{eq:DiscretizedNewton} is 
modified by adding an augmented Lagrangian term
\begin{align}
\tA_{\gamma_d} \coloneqq \tA + \gamma_d \tB^\top \tM_p^{-1} \tB,
\end{align}
where $\gamma_d \gg 0$, and the right-hand side of \cref{eq:DiscretizedNewton} is modified so that the solution of linear system remains unchanged (since $\tB\vectt{u}$ is known). In particular, if the current velocity iterate is divergence-free, then $\gamma_d \tB^\top \tM_p^{-1} \tB \delta \vectt{u} = \vectt{0}$ and no modification to the right-hand side is required. While it does not change the solution, the addition of the augmented Lagrangian term influences the nature of the inner Schur complement. In particular, $\tS_1$ becomes
\begin{align}
\tS_{1, \gamma_d} = 
\begin{pmatrix}
\tA_{\gamma_d} - \tD \tC_\mu^{-1} \tD^\top  & \tB^\top       \\
\tB                                & \tzero               
\end{pmatrix},
\label{eq:Schur1gammad}
\end{align}
and $\tS_2$ becomes
\begin{align}
\begin{split}
\tS_{2,\gamma_d} = -\tB(\tA_{\gamma_d} - \tD \tC_\mu^{-1} \tD^\top)^{-1}\tB^\top.
\label{eq:S2approx}
\end{split}
\end{align}

\begin{remark}
As noted in \cref{rem:Schur-indefinite}, it is possible that $\tA_{\gamma_d} - \tD \tC_\mu^{-1} \tD^\top$ is indefinite and, therefore, it is unclear if $\tS_{1, \gamma_d}$ is invertible. If we were to encounter such a case, then the BM linear system solve would fail. Here, the deflated barrier method would repeatedly halve the step size in the barrier parameter $\mu$ (keeping the initial guess the same). We did not observe this behaviour in our numerical experiments. Henceforth, we assume that $\tS_{1, \gamma_d}$ is invertible. We note that if $\tS_{1,\gamma_d}$ is invertible, then $\tS_{2,\gamma_d}$ is invertible. This is seen by applying Sylvester's law of inertia to the Schur complement decomposition \cref{eq:blockprecon} \cite[Ch.~4]{Elman2014}. 
\end{remark}

The action of $\tS^{-1}_{2,\gamma_d}$ is required during the solve. As already mentioned, $\tS_{2,\gamma_d}$ is dense and, hence, it is expensive to assemble and apply its inverse. In the next proposition we show that the action of $\tS^{-1}_{2,\gamma_d}$ can be approximated with $-\gamma_d \tM^{-1}_p$ with increasing accuracy as $\gamma_d \to \infty$. In practice, values of $\gamma_d \sim \mathcal{O}(10^5)$ were sufficient to approximate the action of $\tS^{-1}_{2,\gamma_d}$ with $-\gamma_d \tM^{-1}_p$ to tolerances of $\mathcal{O}(10^{-10})$. 

\begin{proposition}
\label{prop:Schur-Mp}
Suppose that $\tA - \tD \tC_\mu^{-1} \tD^\top$, $\tA_{\gamma_d} - \tD \tC_\mu^{-1} \tD^\top$, $\tS_2$ and $\tS_{2,\gamma_d}$ are invertible and $\tB$ has been modified to ensure it has full rank (see \cref{rem:B-rank}). Consider the eigenvalue problem
\begin{align}
\gamma_d \tM_p^{-1} \tS_{2,\gamma_d} \vectt{q}_i = \lambda_i \vectt{q}_i, \;\; i = 1,\dots,n_p.
\label{eq:evalue1}
\end{align}
Then, $\lambda_i \to -1$, $i = 1,\dots,n_p$ as $\gamma_d \to \infty$. 
\end{proposition}
\begin{proof}
We note that $\tM_p$ is symmetric positive-definite and, therefore, invertible. Our assumptions satisfy the requirements of \cite[Lem.~4.1]{Benzi2006} and thus
\begin{align}
-\tS^{-1}_{2,\gamma_d} =-\tS^{-1}_{2} + \gamma_d \tM_p^{-1}. \label{eq:evalue2}
\end{align}
By left multiplying \cref{eq:evalue1} by $\gamma_d^{-1} \tM_p$ then $\tS^{-1}_{2,\gamma_d}$ and applying \cref{eq:evalue2} we see that
\begin{align}
\lambda_i \left[ \tS_2^{-1}  \gamma_d^{-1} \tM_p - \tens{I} \right] \vectt{q}_i = \vectt{q}_i. \label{eq:evalue3}
\end{align}
As $ \tM_p^{-1} \tS_{2,\gamma_d} $ is invertible, then $\lambda_i \neq 0$ for all $i = 1,\dots,n_p$. Hence, by rearranging \cref{eq:evalue3} we observe that
\begin{align}
\tS_2^{-1} \tM_p \vectt{q}_i = \gamma_d ( \lambda_i^{-1} + 1) \vectt{q}_i.
\end{align}
Since $\tS_2^{-1} \tM_p$ is invertible, its eigenvalues (denoted $\delta_i$) are nonzero, and satisfy
\begin{align}
\delta_i = \gamma_d (\lambda_i^{-1} +1), \;\; i = 1,\dots,n_p.
\end{align}
Therefore, by taking the limit $\gamma_d \to \infty$ we conclude that $\lambda_i \to -1$ for all $i =  1,\dots,n_p$.
\end{proof}

If assembled na\"ively, the triple matrix product $\tB^\top \tM_p^{-1} \tB$, as it occurs in the augmented Lagrangian term, is expensive to compute. However, it can be verified that the augmented Lagrangian term $\gamma_d \tB^\top \tM_p^{-1} \tB$ corresponds to augmenting the weak form $a_{h}(\vect{u}_h,\vect{v}_h; \rho_h)$ in \cref{dfoc2} by
\begin{align}
\gamma_d \int_\Omega \Pi (\divv{\vect{u}_h}) \Pi (\divv{\vect{v}_h})  \dx,
\label{eq:proj}
\end{align}
where $\Pi$ is the projection onto the discretized pressure space. The projection is the identity for the BDM-DG pair. Therefore, assembling $\tB \tM_p^{-1} \tB^\top$ is equivalent to assembling the matrix associated with the bilinear form $\int_\Omega \divv{\vect{\phi}_i} \divv{\vect{\phi}_j} \dx$, where $\vect{\phi}_i$, $i = 1, \dots, n_{\vect{u}}$, are the basis functions of the velocity finite element space. 

With the proposed nested block preconditioning, we have reduced solving linear systems involving \cref{eq:DiscretizedNewton} into the following steps:
\begin{enumerate}
\itemsep=0pt
\item Apply an outer FGMRES solver;
\item Apply the reciprocal of $\tS_0 \in \mathbb{R}$;
\item Invert the diagonal matrix $\tC_\mu$;
\item Factorize and solve the block-diagonal pressure mass matrix $\tM_p$;
\item Apply the action of the inverse of the augmented momentum block $\tA_{\gamma_d} - \tD \tC_\mu^{-1} \tD^\top$.
\end{enumerate}
Factorizing $\tM_p$, and $\tA_{\gamma_d} - \tD \tC_\mu^{-1} \tD^\top$ with a direct solver such as MUMPS~\cite{mumps} is faster than factorizing the full matrix in \cref{eq:DiscretizedNewton}. We note that $\tM_p$ only needs to be factorized once at the start of the algorithm. With ideal inner solvers, most of the computational time during the run of the deflated barrier method is spent on factorizing $\tA_{\gamma_d} - \tD \tC_\mu^{-1} \tD^\top$ at each BM iteration.

In \cref{fig:flowauglag}, we summarize the block preconditioning strategy for solving linear systems involving the density-momentum-pressure block \cref{rho-u-p-block}. In the next section we develop a specialized geometric multigrid scheme to efficiently solve linear systems involving $\tA_{\gamma_d} - \tD \tC_\mu^{-1} \tD^\top$ (highlighted in pink in \cref{fig:flowauglag}) in order to reduce the computational time further when the problem is discretized on a fine mesh. 
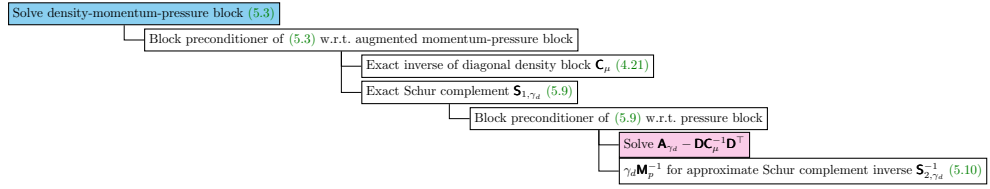
\begin{figure}[ht]
\centering
 \begin{tikzpicture}[%
every node/.style={draw=black, thick, anchor=west},
grow via three points={one child at (0.0,-0.7) and
    two children at (0.0,-0.7) and (0.0,-1.4)},
edge from parent path={(\tikzparentnode.210) |- (\tikzchildnode.west)}]
\node[fill=babyblue]{Solve density-momentum-pressure block  \cref{rho-u-p-block}}
          child{ node{Block preconditioner of \cref{rho-u-p-block} w.r.t.~augmented momentum-pressure block}
                child { node {Exact inverse of diagonal density block $\tC_\mu$ \cref{Cmu}}
                 }     
                 child { node {Exact Schur complement $\tS_{1,\gamma_d}$ \cref{eq:Schur1gammad}}     
		            child { node {Block preconditioner of \cref{eq:Schur1gammad} w.r.t.~pressure block}
		                child { node[fill=classicrose]{Solve $\tA_{\gamma_d} - \tD \tC_\mu^{-1}\tD^\top$}}
		                child { node {$\gamma_d \tM_p^{-1}$ for approximate Schur complement inverse $\tS_{2,\gamma_d}^{-1}$ \cref{eq:S2approx}}
		                }
		            }
                   }
                   };
\end{tikzpicture}
\caption{The preconditioning strategy to solve the density-momentum-pressure block \cref{rho-u-p-block}. We develop a geometric multigrid scheme for the item in pink in \cref{sec:mg}.}\label{fig:flowauglag}
\end{figure}

\subsection{A specialized multigrid scheme for $\tA_{\gamma_d} - \tD \tC_\mu^{-1}\tD^\top$}
\label{sec:mg}

As already mentioned, the tradeoff for using an augmented Lagrangian term to control the Schur complement $\tS_{2,\gamma_d}$ is that $\tA_{\gamma_d} - \tD \tC_\mu^{-1}\tD^\top$ becomes difficult to solve, due to the semi-definite term with a large coefficient $\gamma_d \gg 0$. In the past two decades, there has been progress on specialized multigrid schemes, based on the pioneering work of Sch\"oberl \cite{Schoberl1999}, to handle the effects of the augmented Lagrangian term in $\tA_{\gamma_d}$. Strategies based on Sch\"oberl's work have been shown to be extremely effective in parameter-robust preconditioning of the three-dimensional incompressible Navier--Stokes equations \cite{Farrell2021, Farrell2019a}, Oseen--Frank models of cholesteric liquid crystals \cite{Xia2020}, implicitly-constituted non-Newtonian incompressible flow \cite{Gazca-Orozco2020, Farrell2020d}, and magnetohydrodynamics \cite{Laakmann2021}. Related multigrid schemes have also been analyzed in the context of the $\Hdiv$ and $\Hcurl$ Riesz maps \cite{Arnold2000} and, more relevant to our problem, an $\Hdiv$-conforming discretization of the Stokes equations \cite{Hong2016}. 

In this work, we construct a mesh hierarchy and construct our multilevel hierarchy geometrically. The coarse-level operators are induced from \cref{dfoc1}--\cref{dfoc4} via rediscretization. Moreover, in all our examples, we choose a direct solver for the coarse-level solver. Sch\"oberl's analysis gives sufficient conditions on the relaxation method and transfer operators to achieve robustness in the context of multigrid cycles applied to symmetric positive-definite problems augmented with a parameter-dependent positive semi-definite term. The first is that the relaxation method must stably capture the kernel of the semi-definite term. The second requirement is that the prolongation operator must have a continuity constant that is independent of $\gamma_d$. As noted by Hong et al.~\cite[Sec.~1]{Hong2016}, in a nested mesh hierarchy, an exactly divergence-free function on the coarse-grid will be divergence-free on the fine-grid. Therefore, in our context, the natural prolongation operator suffices thanks to our choice of discretization, and we only discuss the relaxation method in this work. The kernel of the semi-definite term involving $\gamma_d$ is
\begin{align}
\mathcal{N}_h = \{ \vect{w}_h \in \vect{X}_{\mathrm{BDM}_k} : (\divv{\vect{w}_h}, \divv{\vect{v}_h})_{L^2(\Omega)} = 0 \; \text{for all} \; \vect{v}_h \in \vect{X}_{\mathrm{BDM}_k}\}, 
\end{align}
i.e.~all functions with divergence zero. For large $\gamma_d$, $\tA_{\gamma_d}$ becomes increasingly singular. Common relaxation methods like Jacobi and Gauss-Seidel do not offer $\gamma_d$-robust smoothing and yield ineffective multigrid cycles. To understand the degradation of Jacobi and Gauss--Seidel as $\gamma_d \to \infty$, it is fruitful to view the relaxation method as a subspace correction method \cite{Xu1992, Xu2001}. Consider the space decomposition
\begin{align}
\vect{X}_{\mathrm{BDM}_k} = \sum_i \vect{X}_i, \label{spacedecomposition}
\end{align}
where the sum is not necessarily direct. A subspace correction method solves for an approximation of the error in each subspace, and combines them (additively or multiplicatively) to form an updated guess for the solution. In the classical Jacobi and Gauss--Seidel iterations, the decomposition, \cref{spacedecomposition}, is given by $\{\vect{X}_i \} = \{\vect{\phi}_i \}$ where $\vect{\phi}_i$, $i=1,\dots,n_{\vect{u}}$, are the velocity basis functions. The difference between Jacobi and Gauss-Seidel is whether the updates are applied additively (Jacobi) or multiplicatively (Gauss--Seidel). 

A sufficient condition for the subspace correction method induced by the decomposition \cref{spacedecomposition} to be robust in $\gamma_d$ for a symmetric positive-definite matrix, is that the decomposition captures the kernel $\mathcal{N}_h$ in the following sense \cite{Benzi2006, Schoberl1999, Farrell2019a, Hong2016}:
\begin{align}
\mathcal{N}_h = \sum_i \vect{X}_i \cap \mathcal{N}_h. \label{kerneldecomposition}
\end{align}
In other words, the decomposition must be sufficiently rich so that any divergence-free velocity can be written as a combination of divergence-free functions from the subspaces $\vect{X}_i$\footnote{In addition, the decomposition must be stable, but we do not elaborate here.}. Jacobi fails this criterion, as each $\vect{\phi}_i$ is not divergence-free. A decomposition satisfying \cref{kerneldecomposition} for the BDM-DG discretization was developed by Hong et al.~\cite[Sec.~4.5]{Hong2016}, where the decomposition is the so-called \emph{star} patch around every vertex of the mesh. This decomposition is visualized in \cref{fig:starpatch} for a $\mathrm{BDM}_1$ discretization in two dimensions. The same topological decomposition extends to higher orders and three dimensions. Since $\tA_{\gamma_d} - \tD \tC_\mu^{-1} \tD^\top$ is not guaranteed to be positive-definite, the theory does not guarantee robust convergence. Nevertheless, we find that a small number of FGMRES iterations preconditioned with the vertex-star patch iteration is very effective as a smoother, as reported in \cite{Farrell2019a} and subsequent works.
\begin{figure}[ht]
\centering
\includegraphics[width =0.3 \textwidth]{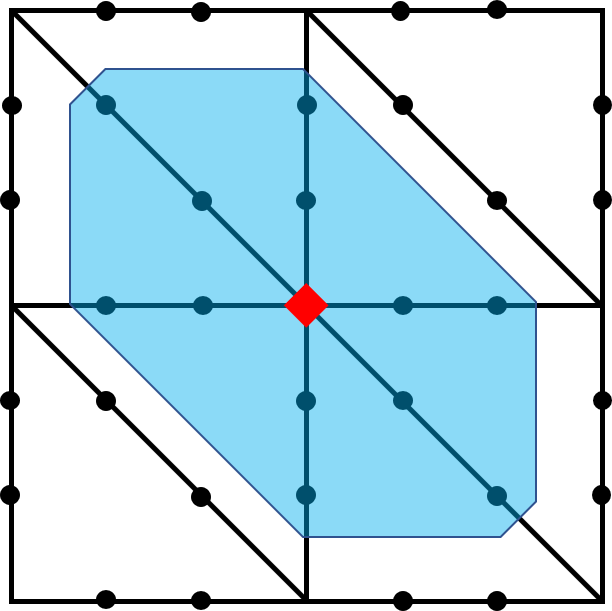}
\caption{The patch of degrees of freedom (black dots inside the blue patch) around a vertex (red diamond) used in the multigrid relaxation for a $\mathrm{BDM}_1$ discretization in two dimensions. Each vertex-star patch contains 12 degrees of freedom in two dimensions with this lowest-order element pair.}\label{fig:starpatch}
\end{figure}

\subsubsection*{Injecting the active set} 
A complication arises in the representation of $\tA_{\gamma_d} - \tD \tC_\mu^{-1} \tD^\top$ on the coarser levels. By first ignoring the BM active set, we note that $\tA_{\gamma_d} - \tD \tC_\mu^{-1} \tD^\top$ can be assembled by injecting the current finite element iterates, $\vect{u}_h$ and $\rho_h$, to the relevant level on the mesh hierarchy, assembling the submatrices $\tA$, $\tD$, $\tD^\top$ and $\tC_\mu$, applying the Dirichlet boundary conditions of the injected velocity to the relevant rows and columns of $\tA$, $\tD$ and $\tD^\top$, and subtracting the triple matrix product $\tD \tC_\mu^{-1} \tD^\top$ from $\tA$. The triple matrix product is sparse and cheap to compute as $\tC_\mu$ is diagonal on all levels. However, we found that an accurate representation of the active set on the coarser levels is essential for the convergence of the multigrid scheme. This is similar to other experiences reported in the literature \cite{Engel2011,Hoppe1987}. Hence, the difficulty lies in defining the active set on the coarser levels. An obvious choice is to use the definitions of $\mathcal{A}$ and $\mathcal{I}$ in \cref{alg:bm} defined via the injected material distribution iterate. However, in numerical experiments we found that this choice resulted in poor iteration counts.

Consider a two-grid method with the fine-level triangulation $\mathcal{T}_h$, $h=H/2$, obtained by a uniform refinement of the simplices in coarse-level triangulation $\mathcal{T}_H$. As the material distribution is discretized with $\DG_0$ elements, each degree of freedom $i$ associated with the fine-level material distribution iterate can be associated with an element $K_h \in \mathcal{T}_h$ in the fine level and analogously with the degrees of freedom of the coarse-level material distribution iterate with elements in the coarse level. We say that a fine-level element $K_h \in \mathcal{T}_h$ is in the active set $\mathcal{A}_h$ (written as $K_h \in \mathcal{A}_h$) if the degree of freedom associated with $K_h$ is in the active set $\mathcal{A}_h$. This definition naturally extends to the coarse-level elements and active sets. 

We now utilize an idea inspired by the work of Hoppe \cite{Hoppe1987} and Engel and Griebel \cite{Engel2011} to define the coarse-level active sets. A coarse-level element, $K_{H} \in \mathcal{T}_H$ containing the parent fine-level elements $K_{h,1},\dots, K_{h,s} \in \mathcal{T}_h$ is defined to be in the coarse-level active set $\mathcal{A}_H$ if 
\begin{align}
|\{K_{h,j}\in \mathcal{A}_h : j = 1, \dots, s \}| \geq m,
\label{coarse-active set}
\end{align}
where $m \in [1,s]$ and $s=4$ in two dimensions and $s = 8$ in three dimensions. In other words, the coarse-level element is in the coarse-level active set if it contains $m$ or more fine-level parent elements that are in the fine-level active set. By starting at the finest-level active set that is defined by \cref{alg:bm}, we recursively define all the active sets in mesh hierarchy via \cref{coarse-active set}. Experiments revealed that a good choice for $m$ is $m = s/2$, i.e.~a coarse-level element is active if at least half of its parent fine-level elements are active. A summary of the multigrid strategy is given in \cref{fig:flowmg}. 
\begin{figure}[ht]
\centering
\includegraphics[width =0.7 \textwidth]{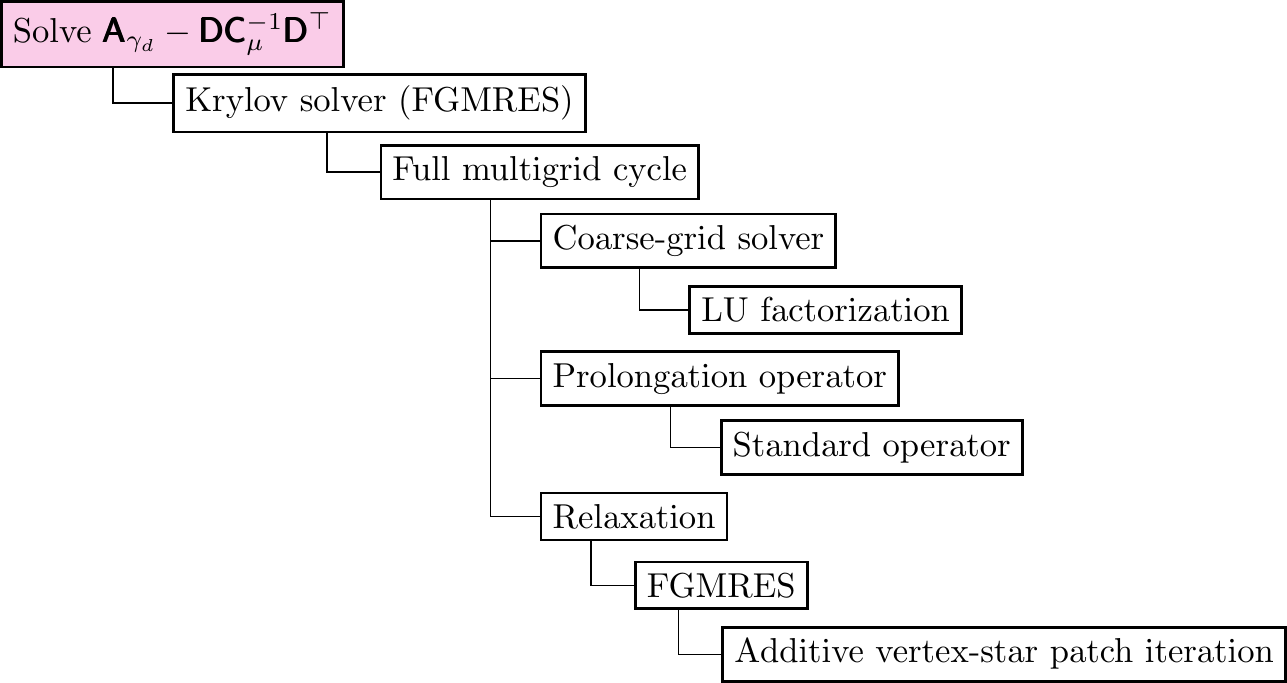}
\caption{The multigrid solver strategy of \cref{sec:mg} to solve $\tA_{\gamma_d} - \tD \tC_\mu^{-1} \tD^\top$.}\label{fig:flowmg}
\end{figure}

\begin{remark}
The choice of \cref{coarse-active set} is more generous than utilizing the definition of the fine-level active set directly with the injected material distribution iterate. In particular, some coarse cells that are ``borderline'' between the active set and inactive set are placed in the active set by  \cref{coarse-active set} but in the inactive set when defining the active set using the injected material distribution. Numerically, we see that the iteration counts suffer if the criteria for a coarse cell to be in the coarse-level active set are too strict. 
\end{remark}

\section{Numerical results}
\label{sec:examples}
All examples in this chapter were implemented with the finite element software Firedrake \cite{firedrake}. Block preconditioning and Krylov subspace methods were implemented using Firedrake \cite{firedrake} and PETSc \cite{petsc}, and sparse LU factorizations were performed with MUMPS \cite{mumps}. Vertex-star patch relaxation is implemented via the PCPATCH functionality  \cite{pcpatch} recently introduced to PETSc. The meshes were created in Firedrake or Gmsh \cite{Geuzaine2009}. The uniqueness of the pressure was enforced by orthogonalizing against the nullspace of constants in the Krylov method. The coarsest-level correction in the multigrid scheme of \cref{sec:mg} is computed via an LU factorization. The BM updates are scaled with a (damped) $l^2$-minimizing linesearch \cite[Alg.~2]{Brune2015} and we do not use a prediction step in any examples. Wherever (F)GMRES is used, it is not restarted. 

\subsection{Double-pipe}
\label{sec:double-pipe}
The first example is the two-dimensional double-pipe problem first introduced by Borrvall and Petersson \cite[Sec.~4.5]{Borrvall2003}. The double-pipe problem is posed on a rectangular domain $\Omega = (0,3/2) \times (0,1)$ with two inlets and two outlets fixed by the following Dirichlet boundary condition
\begin{align}
\label{eq:doublepipebcs}
 \vect{g}(x,y) = 
\begin{cases}
\left(1-144(y-3/4)^2, 0\right)^\top & \text{if} \;\; 2/3 \leq y \leq 5/6, x = 0 \; \text{or} \; 3/2,\;\; \\
\left(1-144(y-1/4)^2, 0\right)^\top & \text{if} \;\; 1/6 \leq y \leq 1/3, x = 0 \; \text{or} \; 3/2,\;\; \\
(0,0)^\top & \text{elsewhere on} \;\partial \Omega.
\end{cases}
\end{align}
We choose a volume fraction of $\gamma = 1/3$ and the inverse permeability $\alpha$ is given in \cref{eq:alphachoice}, with $\bar \alpha = 2.5 \times 10^4$ and $q=1/10$. The problem supports two minima: a local minimum of two straight channels from each inlet to its opposite outlet, and a global minimum in the shape of a double-ended wrench. These are depicted in \cref{fig:double-pipe}.

\begin{figure}[ht]
\centering
\includegraphics[width = 0.3\textwidth]{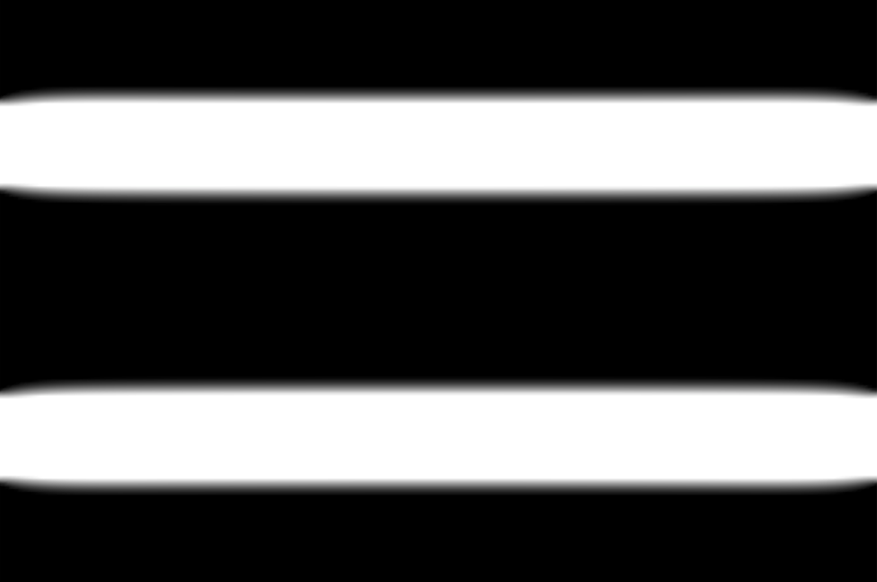}
\includegraphics[width = 0.3\textwidth]{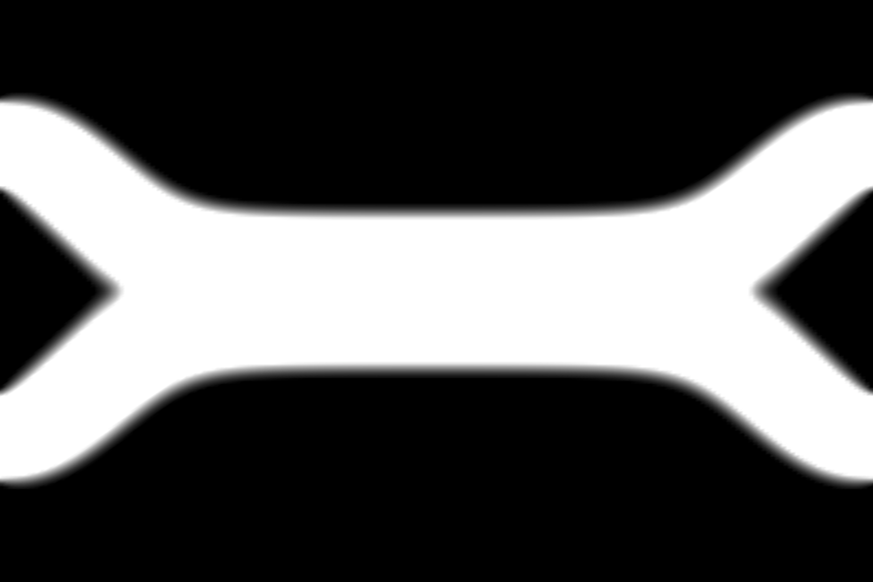}
\caption{The material distribution of the straight-channel (left) and double-ended wrench (right) solutions of the double-pipe optimization problem.} 
\label{fig:double-pipe}
\end{figure}

The two strategies we utilize for solving the linear systems are the following:
\begin{enumerate}[label=({aL}\arabic*)]
\item The nested block preconditioning approach of \cref{sec:explore} on \cref{eq:DiscretizedNewton} with $\gamma_d = 10^4$, and an LU factorization for the augmented momentum block $\tA_{\gamma_d} - \tD \tC_\mu^{-1} \tD^\top$; \label{dp:s2}
\item The nested block preconditioning approach of \cref{sec:explore} on \cref{eq:DiscretizedNewton} with $\gamma_d = 10^4$, and the geometric multigrid method of \cref{sec:mg} to approximate the action of the inverse of $\tA_{\gamma_d} - \tD \tC_\mu^{-1} \tD^\top$. We fix the relaxation to 5 FGMRES iterations preconditioned with a vertex-star patch iteration and a full multigrid cycle is used. \label{dp:s3}
\end{enumerate}

We opt for a first-order Brezzi--Douglas--Marini $\mathrm{BDM}_1 \times \mathrm{DG}_0$ mixed finite element discretization for the velocity-pressure pair, with interior penalty parameter $\sigma=10$, and a $\mathrm{DG}_0$ discretization for the material distribution. This choice of discretization makes both $\tM_p$ and $\tC_\mu$ diagonal. For all mesh sizes, we initialize the deflated barrier method at $\mu_0 = 105$ and perform deflation immediately to find the second branch. The first and second branches converge to the straight channels and double-ended wrench solutions, respectively, as $\mu \to 0$. The nonlinear solves are terminated with an absolute tolerance of $10^{-5}$. The outer FGMRES solver's absolute and relative tolerances are both set to $10^{-7}$. In Tables \labelcref{tab:double-pipe-its-aL1} and \labelcref{tab:double-pipe-its-aL2}, we list the iteration counts for the strategies \labelcref{dp:s2} and \labelcref{dp:s3} on meshes with decreasing mesh sizes. In the \labelcref{dp:s3} strategy, the augmented block solve is approximated to an absolute tolerance of $10^{-8}$ or a relative tolerance of $10^{-9}$. 
\begin{table}[ht]
\small
\centering
\begin{tabular}{ll|ll}
\toprule
&  & \multicolumn{2}{c}{\labelcref{dp:s2}} \\ \midrule
$h$ & Dofs   &  BM & OK \\ \midrule
0.0361&25,201   &	274 	& 461  (1.68)\\
0.0180& 100,401 &  626  &1180 (1.88)\\
0.0090& 400,801  & 733 & 1283 (1.75)\\
0.0045& 1,601,601  & 809	 & 1607 (1.99)\\
\bottomrule
\end{tabular}
\caption{The total cumulative number of iterations to compute both minimizers of the double-pipe problem over all the subproblems with the \labelcref{dp:s2} preconditioner. BM stands for the number of Benson--Munson iterations and OK stands for the number of outer Krylov FGMRES iterations. The numbers in brackets in the OK  column are the number of average Krylov iterations per BM iteration.} 
\label{tab:double-pipe-its-aL1}
\end{table}

\begin{table}[ht]
\small
\centering
\begin{tabular}{ll|lll}
\toprule
&  &\multicolumn{3}{c}{\labelcref{dp:s3} 2-grid}  \\ \midrule
$h$ & Dofs   & BM & OK & IK \\ \midrule
0.0180& 100,401 & 627 &1270 (2.03)&100,109 (13.14)\\
0.0090& 400,801  & 735 & 1360 (1.85) & 86,745 (10.63)\\
\bottomrule
\end{tabular}
\caption{The total cumulative number of iterations to compute both minimizers of the double-pipe problem over all the subproblems with the \labelcref{dp:s3} preconditioner. BM stands for the number of Benson--Munson iterations, OK stands for the number of outer Krylov FGMRES iterations, and IK is the number of inner Krylov FGMRES iterations preconditioned with the geometric multigrid method of \cref{sec:mg} to solve linear systems involving $\tA_{\gamma_d} - \tD \tC_\mu^{-1} \tD^\top$. The numbers in brackets in the OK and IK columns are the number of average Krylov iterations per BM iteration and per $\tA_{\gamma_d} - \tD \tC_\mu^{-1} \tD^\top$ solve, respectively.} 
\label{tab:double-pipe-its-aL2}
\end{table}

We see that the Krylov iterations per BM iteration are robust to the mesh size for both preconditioning strategies. The preconditioning strategy with an LU factorization for the augmented momentum block \labelcref{dp:s2} averages to under 2 preconditioned FGMRES iterations per BM iteration. Similarly with \labelcref{dp:s3}, where the augmented block solve is approximated with FGMRES preconditioned with a 2-grid multigrid cycle, the outer FGMRES iterations remain under 2  preconditioned FGMRES iterations per BM iteration on average. Moreover, the average inner FGMRES based on the kernel-preserving multigrid scheme iterations remain under 13 iterations per augmented momentum block solve over all mesh sizes. In particular, the average inner FGMRES iterations decreases on the fine mesh relative to the coarse mesh. We note that, unlike the conforming discretizations in \cite{Papadopoulos2021a}, the number of BM iterations slowly increases with decreasing mesh size. This may be due to the discretization or could be related to the fact that the linear systems are not being solved exactly. We do not report the timings of the solves in this example. In two dimensions a direct solve, using MUMPS with $\gamma_d = 0$, of the full BM system \cref{eq:DiscretizedNewton} is faster than the block preconditioning and multigrid strategy proposed here. However, we note that in three dimensions, a direct solve often fails due to the ill-conditioning of the system. Moreover, assembling fine-grid matrices has large memory requirements. As the \labelcref{dp:s3} strategy only requires an assembly of the coarsest-grid augmented block, this allows for finer mesh linear solves. This is further discussed in \cref{sec:3d-5-holes}.

A natural question is the sensitivity of the solver with respect to the augmented Lagrangian parameter $\gamma_d$. We wish to explore the following:
\begin{enumerate}[label=({H}\arabic*)]
\itemsep=0pt
\item The effect of roundoff error as $\gamma_d$ increases; \label{hyp:1}
\item The improvement of the approximation of $-\gamma_d \tM_p^{-1}$ for $\tS_{2,\gamma_d}^{-1}$; \label{hyp:2}
\item The robustness of the vertex-star patch multigrid cycle for increasing $\gamma_d$. \label{hyp:3}
\end{enumerate}
An unfortunate side-effect of increasing $\gamma_d$ is the expectation of high roundoff error. This would manifest in the nonlinear solver stagnating at higher residual norm values than desired. However, as $\gamma_d$ increases, we expect that the Schur complement approximation by the pressure mass matrix improves as shown in \cref{prop:Schur-Mp}. Hence, the number of outer Krylov FGMRES iterations per BM iteration should decrease as $\gamma_d \to \infty$. Finally, the multigrid cycle should be robust to the value of $\gamma_d$. Hence, we expect that the number of inner Krylov FGMRES iterations preconditioned by the multigrid cycle to approximate the inverse of the augmented momentum block should stay roughly constant with increasing $\gamma_d$. In \cref{tab:double-pipe-its-gammad} we test these hypotheses. 

We investigate \labelcref{hyp:1} and \labelcref{hyp:2} via the \labelcref{dp:s2} solver (with varying $\gamma_d$). We fix a mesh with mesh size $h=0.0361$ and choose the discretization used previously. We consider the deflated barrier method subproblem at $\mu=105$ and $\mu = 1$ for the first branch. The nonlinear solves are terminated when the decrease in the residual norm stagnates. The outer FGMRES solver's absolute and relative tolerances are set to $10^{-15}$ and $10^{-5}$, respectively. The middle four columns confirm our first two hypotheses. At $\gamma_d = 10^7$ we achieve the optimal number of outer FGMRES iterations per BM iteration. The approximation of the Schur complement with the pressure mass matrix in near-perfect. However, we note that for each increase in the order of magnitude after $\gamma_d = 10^2$, the roundoff error causes the nonlinear solver to stagnate at an order of magnitude higher. We believe a good compromising choice is $\gamma_d = 10^4$.

To investigate \labelcref{hyp:3} we utilize the \labelcref{dp:s3} solver (with varying $\gamma_d$). We fix a mesh with mesh size $h=0.0361$ for the coarse level and uniformly refine the mesh once for the fine level. We consider the first outer Krylov FGMRES iteration at $\mu=105$ and tabulate the average number of multigrid preconditioned inner Krylov FGMRES iterations per augmented momentum block solve required. The augmented momentum block solve is approximated to an absolute tolerance of $10^{-8}$ or a relative tolerance of $10^{-9}$. We see that with increasing $\gamma_d$ the iterations stay roughly constant. This confirms the cycles are robust in $\gamma_d$. 
\begin{table}[ht]
\small
\centering
\begin{tabular}{l|ll|ll|l}
\toprule
& \multicolumn{4}{c}{\labelcref{dp:s2}}  & \labelcref{dp:s3}  \\ \midrule
& \multicolumn{2}{c}{$\mu=105$}  & \multicolumn{2}{c}{$\mu=1$}  & $\mu=105$\\ \midrule
$\gamma_d$ & avg.~OK  & res.~norm & avg.~OK  & res.~norm & avg.~IK   \\ \midrule
$10^{-5}$ &  157.2 &	$9.90\times10^{-12}$	& 66.0  & $6.40\times10^{-12}$ & 10.0   \\
$10^{-1}$ &133.5   &$9.93\times10^{-12}$& 93.9  & $6.73\times10^{-12}$& 11.3  \\
$10^{0}$ & 77.8 &$1.09\times10^{-11}$& 40.6  & $6.92\times10^{-12}$& 12.8 \\
$10^{1}$ & 30.0 &$1.15\times10^{-11}$& 16.3  & $8.99\times10^{-12}$& 10.0  \\
$10^{2}$ &8.23  &$3.92\times10^{-11}$	&  4.44 & $4.73\times10^{-11}$& 9.50 \\
$10^{3}$ &3.31  &$3.75\times10^{-10}$	& 3.12   & $4.67\times10^{-10}$& 9.17  \\
$10^{4}$ &1.85   &$3.71\times10^{-9}$&  2.20 & $4.57\times10^{-9}$&  9.50 \\
$10^5$ &  1.33 &	$3.70\times10^{-8}$	& 1.67  &$4.68\times10^{-8}$&  9.50  \\  
$10^6$ &  1.25 &	$3.77\times10^{-7}$	& 1.27  &$4.57\times10^{-7}$ &  9.50   \\  
$10^7$ &  1.14 &	$3.73\times10^{-6}$	&1.20 & $4.59\times10^{-6}$ & 9.50     \\ 
$10^8$ &  1.18  &	$3.63\times10^{-5}$	& 1.50   & $4.92\times10^{-5}$& 9.50 \\
\bottomrule
\end{tabular}
\caption{The sensitivity of the \labelcref{dp:s2} and \labelcref{dp:s3} solver strategies to the value of $\gamma_d$. For the \labelcref{dp:s2} strategy, the columns labelled avg.~OK are the average number of outer Krylov FGMRES iterations per BM iteration to solve the deflated barrier method subproblem at $\mu=105$ and $\mu=1$ for $h=0.0361$. The columns labelled res.~norm hold the values of the smallest residual norm before the nonlinear solver stagnates due to roundoff errors. For the \labelcref{dp:s3} strategy the coarse level has a mesh size of $h=0.0361$ and we perform one uniform refinement. The column labelled avg.~IK is the average number of multigrid preconditioned inner Krylov FGMRES iterations per augmented momentum block solve as seen in the first outer Krylov FGMRES iteration at subproblem $\mu=105$.} 
\label{tab:double-pipe-its-gammad}
\end{table}

\subsection{3D cross-channel} 
The first three-dimensional example we consider is the cross-channel problem as found in S\'a et al.~\cite[Sec.~7.5]{Sa2016}. The domain is the unit cube, $\Omega = (0,1)^3$, with two circular inlets and two circular outlets that are arranged in a cross pattern as visualized in \cref{fig:cross-channel-setup}. The volume fraction is given by $\gamma = 1/10$ and we use \cref{eq:alphachoice} as our choice of $\alpha$, with $\bar \alpha = 2.5 \times 10^4$ and $q=1/10$. The Dirichlet boundary condition on $\vect{u}$ is
\begin{align}
\vect{g}(x,y,z) = \left(1-12\pi((y-a)^2+(z-b)^2), 0, 0\right)^\top,
\end{align}
if $12\pi((y-a)^2+(z-b)^2) \leq 1$ and $x = 0$ with $a = 1/2$, $b \in \{1/4, 3/4 \}$ or $x = 1$ with $a \in \{1/4, 3/4 \}$, $b= 1/2$, and $\vect{g}(x,y,z) = (0,0,0)^\top$ elsewhere on $\partial \Omega$.

\begin{figure}[ht]
\centering
\includegraphics[width = 0.35\textwidth]{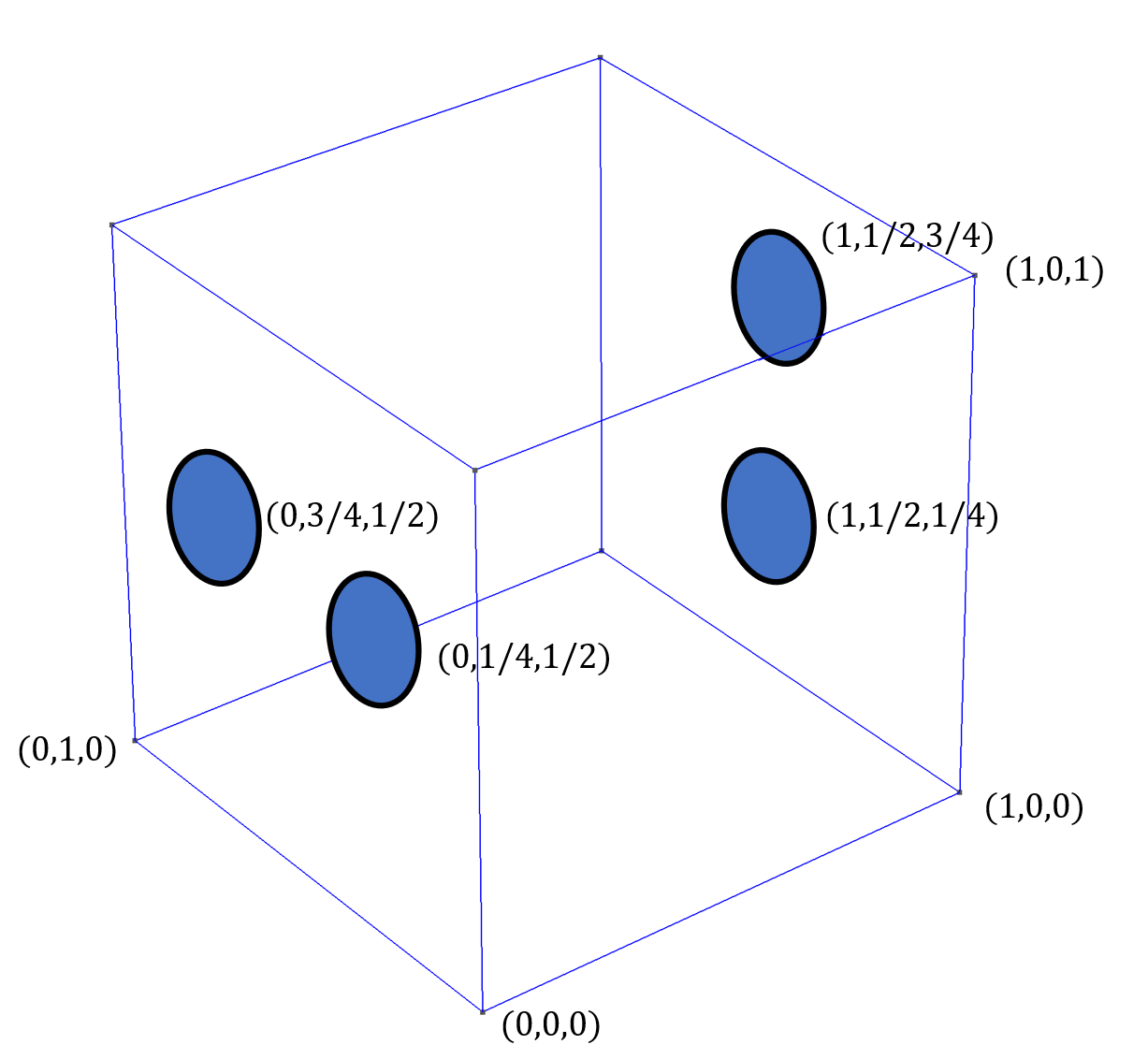}
\caption{Setup of the 3D cross-channel problem. This problem features a unit cube domain with two inlets and two outlets arranged in a cross pattern.}\label{fig:cross-channel-setup}
\end{figure}

We apply the same first-order BDM discretization, with interior penalty penalization parameter $\sigma=10$, and run the deflated barrier method twice. The first pass is on a $20 \times 20 \times 20$ mesh resulting in 391,201 degrees of freedom. The augmented Lagrangian parameter is chosen to be $\gamma_d = 10^6$. Due to the nested block preconditioning, the action of the inverse of the augmented momentum block $\tA_{\gamma_d} - \tD \tC_\mu^{-1}\tD^\top$ must be applied six times per outer FGMRES iteration. On this relatively coarse mesh, it is cheaper to factorize $\tA_{\gamma_d} - \tD \tC_\mu^{-1} \tD^\top$ with MUMPS at the start of each BM iteration and reuse the factorization, rather than iteratively solve $\tA_{\gamma_d} - \tD  \tC_\mu^{-1}\tD^\top$ with multigrid each time an inverse action is required, i.e.~we use the augmented Lagrangian preconditioner \labelcref{dp:s2} for the linear systems. We initialize the barrier parameter at $\mu_0 = 100$. The nonlinear solves are terminated with an absolute tolerance of $10^{-5}$ and the outer FGMRES solver's absolute tolerance is set to $5 \times 10^{-7}$ (all linear solves converged to the absolute tolerance).  A second branch of solutions is found at $\mu = 38.74$ and a third branch at $\mu = 34.87$. 

These coarse-mesh solutions at $\mu = 0$ are interpolated onto a finer $40 \times 40 \times 40$ mesh resulting in 3,100,801 degrees of freedom. The deflated barrier method is then reinitialized at $\mu_0 = 10^{-6}$; we found that the BM solver often diverges if initialized at $\mu_0=0$. On this finer mesh, we again apply the augmented Lagrangian preconditioner ($\gamma_d = 10^5$). The nonlinear solves are terminated with an absolute tolerance of $10^{-4}$ if $\mu > 0$ and $10^{-6}$ if $\mu = 0$. Moreover, the outer FGMRES solver's absolute tolerance is set to $10^{-7}$ if $\mu > 0$ and $10^{-9}$ if $\mu =0$ and its relative tolerance is set to $10^{-7}$. Now a direct solve of the augmented momentum block is prohibitive and we switch to the \labelcref{dp:s3} strategy where each approximate inverse of the augmented momentum block is solved to an absolute or relative tolerance of $10^{-8}$ or $10^{-9}$, respectively, with the (2-grid) multigrid scheme of \cref{sec:mg}. For the relaxation on the fine level, we use 5 FGMRES iterations preconditioned with the vertex-star patch relaxation. 

The resulting three solutions are shown in \cref{fig:cross-channel}. Two of these are symmetric straight channel solutions where the inlets swap which outlet they exit from. Their symmetry results in similar costs. A third global minimizer comes in the form of a merged channel solution; the two channels briefly merge in the middle of the box domain before splitting to exit via the two outlets. 

The iteration counts for the initial search on a $20 \times 20 \times 20$ mesh and the refinement on a  $40 \times 40 \times 40$ mesh are given in \cref{tab:cross-channel}. We see that our preconditioner is effective in both cases. When using a direct solve of the augmented momentum block, we average slightly more than one outer FGMRES iterations per BM iteration. Similarly, the tolerances for the augmented momentum block solve on the fine mesh are strict enough so that the outer FGMRES iterations average between 2.24 and 2.55. Moreover, each augmented momentum block solve requires an average in the range of 15.77--16.77 multigrid preconditioned FGMRES iterations to reach the prescribed tolerances. 
\begin{figure}[ht]
\centering
\subfloat[Branch 0]{\includegraphics[width = 0.32\textwidth]{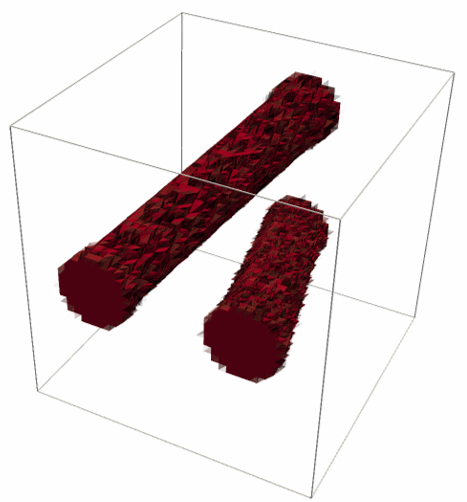}}
\subfloat[Branch 1]{\includegraphics[width = 0.32\textwidth]{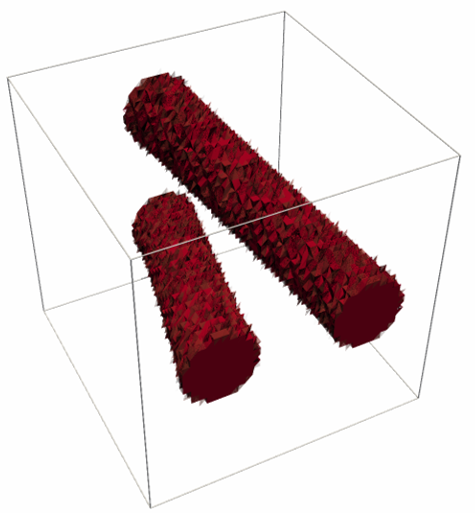}}
\subfloat[Branch 2]{\includegraphics[width = 0.32\textwidth]{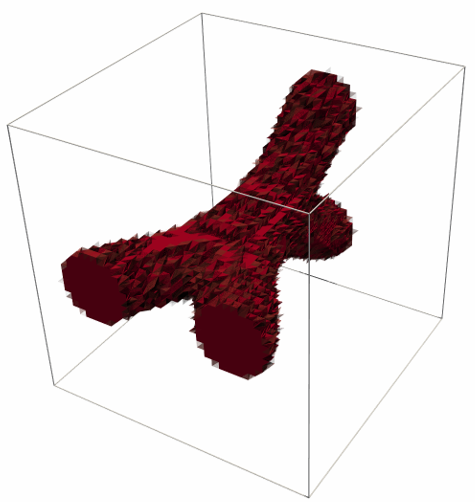}}
\caption{The material distribution of the solutions discovered by the deflated barrier method to the 3D cross-channel optimization problem discretized with 3,100,801 degrees of freedom. The power dissipation values are $J_h = 14.51, 14.62$, and $13.08$ for branches 0, 1, and 2, respectively.}
\label{fig:cross-channel}
\end{figure}

\begin{table}[ht]
\centering
\begin{tabular}{l|ll|lll}
\toprule
& \multicolumn{2}{c|}{Coarse mesh, $\gamma_d = 10^6$} & \multicolumn{3}{c}{Fine mesh, $\gamma_d = 10^5$} \\
\midrule
Branch & BM & OK & BM & OK  & IK \\ \midrule
0& 	828 & 829 (1.00) & 49& 110  (2.24) &	10,405 (15.77) \\
1& 444 & 445 (1.00) & 52 &  121 (2.33)   & 12,176	(16.77)    \\
2&  409& 422 (1.03) & 53 & 135 (2.55)  &13,085 (16.15)\\  
\bottomrule
\end{tabular}
\caption{Cumulative number of BM iterations, outer FGMRES iterations (OK), and for the fine mesh, inner FGMRES iterations preconditioned with the multigrid scheme of \cref{sec:mg} (IK) for the 3D cross-channel problem. The bracketed numbers in the OK and IK columns are the average number of outer FGMRES iterations per BM iteration and average number of inner FGMRES iterations per augmented momentum block solve, respectively. The barrier parameter is initialized at $\mu_0 =100$ on the coarse mesh and $\mu_0=10^{-6}$ on the fine mesh.} 
\label{tab:cross-channel}
\end{table}

\subsection{3D five-holes quadruple-pipe}
\label{sec:3d-5-holes}
In \cite[Sec.~4.4]{Papadopoulos2021a}, it was observed that introducing holes in a rectangular domain caused a significant increase in the number of solutions. We now extend this idea to three dimensions and introduce the generalization of the five-holes double-pipe problem. This problem features a box domain  $\Omega = (0,3/2) \times (0,1) \times (0,1)$ with five internal holes in the shape of cubes, of edge length $1/10$, with centres at $(3/4,1/4,1/4)$, $(3/4,1/4,3/4)$, $(3/4,3/4,1/4)$, $(3/4,3/4,3/4)$, and $(3/4,1/2,1/2)$. There are four inlets and four outlets. The circular inlets of radius $1/\sqrt{12 \pi}$ are positioned on the face $x = 0$ with the centres $(y,z) = (1/4,1/4)$, $(1/4,3/4)$, $(3/4,1/4)$, and $(3/4,3/4)$. The circular outlets of the same radius are positioned on the face $x =3/2$ with the same centres. The domain setup is depicted in \cref{fig:3d-5-holes-setup}. We impose a parabolic Dirichlet boundary condition on the inlets and outlets and a zero Dirichlet boundary condition elsewhere on the boundary (including the boundary of the five internal holes), i.e.~the Dirichlet boundary condition is given by:
\begin{align}
\label{eq:quadpipebcs}
\vect{g}(x,y,z) = \left(1-12\pi((y-a)^2+(z-b)^2), 0, 0\right)^\top,
\end{align}
if $12\pi((y-a)^2+(z-b)^2) \leq 1$, where $ a,b \in \{1/4, 3/4 \}, \; x = 0 \; \text{or} \; 3/2$ and
\begin{align}
\label{eq:quadpipebcs2}
\vect{g}(x,y,z) = (0,0,0)^\top,
\end{align}
elsewhere on $\partial\Omega$, including the boundaries of the five internal holes. We choose a volume fraction of $\gamma = 1/5$ and the inverse permeability term $\alpha$ is given by \cref{eq:alphachoice}, with $\bar \alpha = 2.5 \times 10^4$ and $q=1/10$.
\begin{figure}[ht]
\centering
\includegraphics[width = 0.4\textwidth]{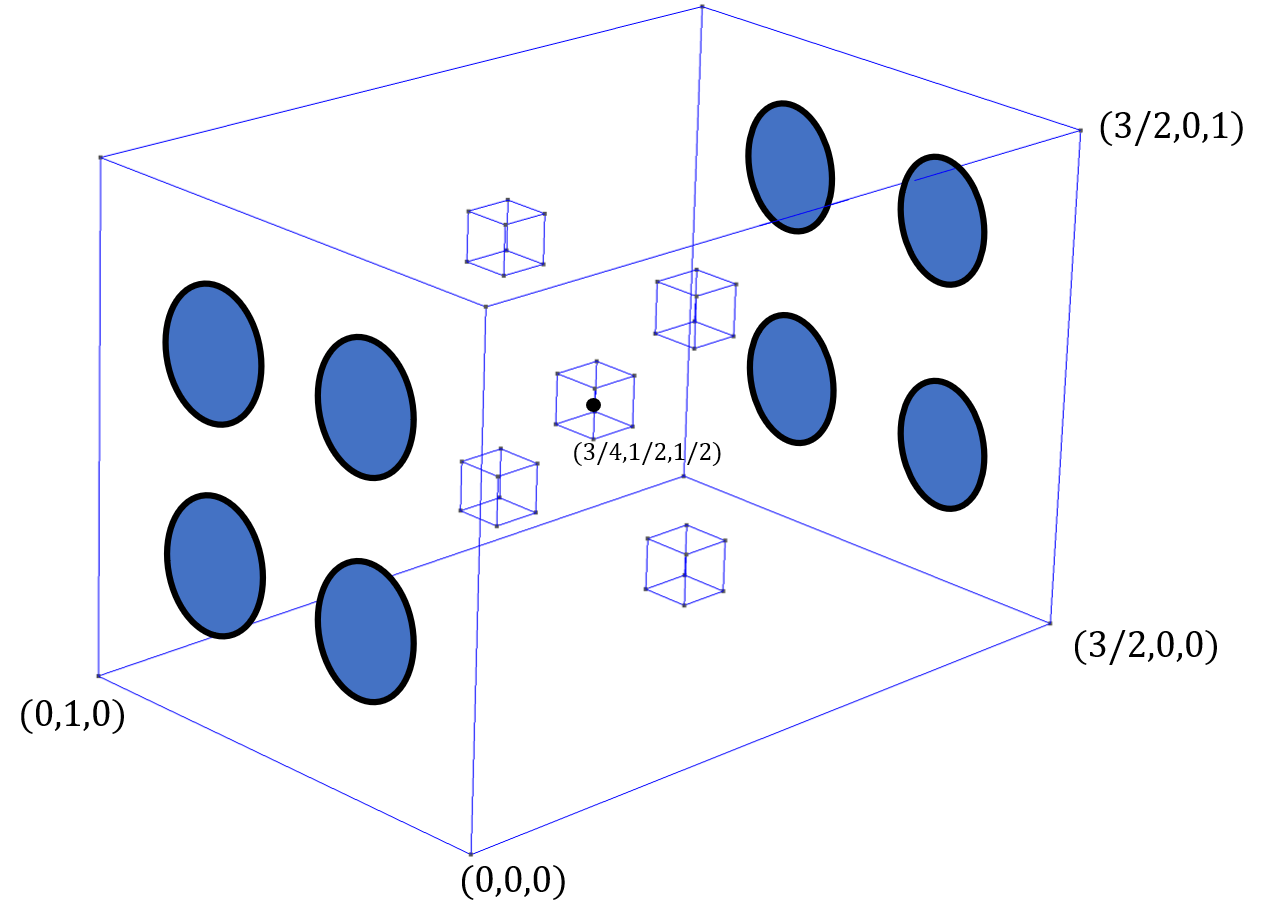}
\caption{Setup of the 3D five-holes quadruple-pipe problem. This problem features 4 inlets and 4 outlets. The domain is a box, $\Omega = (0,3/2) \times (0,1) \times (0,1)$, with five internal holes in the shape of cubes with edge length $1/10$ that are centred at $(3/4,1/4,1/4)$, $(3/4,1/4,3/4)$, $(3/4,3/4,1/4)$, $(3/4,3/4,3/4)$, and $(3/4,1/2,1/2)$.}\label{fig:3d-5-holes-setup}
\end{figure}

The domain is nonconvex and has a complicated geometry due to the cuboidal internal holes, which we do not mesh. We build the coarse level mesh with Gmsh; the finer levels are built by uniform refinement of the coarse mesh. We choose a first-order BDM discretization for the velocity-pressure pair, with interior penalty penalization parameter $\sigma=10^3$, and run the deflated barrier method twice. The larger choice for $\sigma$ is required to sufficiently enforce the boundary conditions in the tangential directions. The first run is on a (relatively) coarse mesh with 30,848 elements which results in 256,745 degrees of freedom. The barrier parameter is initialized at $\mu_0= 200$ and we use the augmented Lagrangian preconditioner \labelcref{dp:s2} for the linear systems, with an augmented Lagrangian parameter value of $\gamma_d = 10^5$. The nonlinear solves are terminated with an absolute tolerance of $10^{-4}$. The outer FGMRES solver's absolute tolerance is set to $5 \times 10^{-7}$ (all linear solves converged to the absolute tolerance).

In total we find 14 coarse-grid solutions. Branches 1 and 2 are found at $\mu=53.81$. Branches 3, 4, 5, and 6 are found at $\mu=11.39$. Branch 7 is found at $\mu=10.25$. Branch 8 is found at $\mu=9.23$. Branch 9 is found at $\mu=6.73$, branch 10 at $\mu = 6.73$, branches 11 and 12 are found at $\mu  = 6.05$, and, finally, branch 13 is found at $\mu = 4.41$. 

We interpolate the coarse-level solutions onto the first refinement, which results in 2,014,113 degrees of freedom. The deflated barrier method is initialized at $\mu_0 =10^{-5}$, using the coarse-level solutions as initial guesses. We damp the $l^2$-minimizing linesearch in the nonlinear solver by a factor of 1/2. The nonlinear solves are terminated with an absolute tolerance of $10^{-4}$ if $\mu > 0$ and $10^{-6}$ if $\mu = 0$. The outer FGMRES solver's absolute tolerance is set to $10^{-6}$ if $\mu > 0$ and $10^{-7}$ if $\mu = 0$, and its relative tolerance is set to $10^{-5}$. We apply an augmented Lagrangian 2-grid multigrid preconditioner \labelcref{dp:s3} to the linear systems with 5 FGMRES iterations for the relaxation of the fine level and $\gamma_d = 10^5$. The augmented momentum block is solved to an absolute or relative tolerance of $10^{-4}$ or $10^{-7}$, respectively. Of the original 14 solutions, the nonlinear solver successfully converges to 7 fine-grid solutions. 

The resulting iteration counts are given in \cref{tab:five-holes-quad-pipe}, the resulting fine mesh solutions are shown in \cref{fig:3d-5-holes-1}, and their cross sections are shown in \cref{fig:3d-5-holes-slice}. A further mesh refinement of branch 0 is shown in \cref{fig:3d-5-holes-2} as computed using the augmented Lagrangian \labelcref{dp:s3} with a 3-grid multigrid cycle. As expected, the five holes obstruct the channels and prevent a large channel passing through the centre. The best solutions found are branches 0, 1, and 11 where the channels form one large channel and either move to the left, upwards or downwards to avoid the middle internal hole. As in the five-holes double-pipe example, there are remaining solutions that we have not yet computed, by symmetry.

In \cref{tab:five-holes-quad-pipe-timings} we give timings for the solvers as run on a machine with 512 GB of RAM and 32 CPUs Intel(R) Xeon(R) CPU E5-4627 v2 @ 3.30GHz. We note that a direct LU factorization of the full BM system \cref{eq:DiscretizedNewton} fails to converge due to the ill-conditioning of the system (even with $\gamma_d = 0$). Thus the preconditioning strategies are necessary to find the BM updates. On the coarse mesh, the \labelcref{dp:s2} strategy is faster than the \labelcref{dp:s3} strategy. Moreover, we show that the average time taken per BM iteration drops from 277 seconds with 1 CPU to 35 seconds when using 32 CPUs. On the fine mesh, the assembly of the full fine-grid augmented block requires more memory than what is available on the workstation (512 GB). Hence only the \labelcref{dp:s3} strategy succeeded in grid-sequencing the solutions on this machine.

\begin{table}[ht]
\small
\centering
\begin{tabular}{l|ll|lll}
\toprule
& \multicolumn{2}{c|}{Coarse mesh} & \multicolumn{3}{c|}{Fine mesh}  \\
\midrule
Branch & BM & OK & BM & OK  & IK \\ \midrule
0& 438& 438  (1) & 41& 109 (2.66) &15,828 (24.2) \\
1& 338 & 338 (1) & 36& 97 (2.69) &13,736 (23.6) \\
2& 396 & 396 (1)   &49 &130 (2.65) &18,999 (24.4) \\
3& 254 & 254 (1)   &46 &114 (2.48)& 16,464 (24.1) \\
7&  221 & 221 (1)   &53&132 (2.49) & 20,061 (25.3) \\
11& 183 & 183 (1)  &64 &180 (2.81) & 27,738 (25.7)\\
13& 187 & 187 (1)  & 60&169 (2.82)& 28,542 (28.1)\\
\bottomrule
\end{tabular}
\caption{Cumulative number of BM iterations, outer FGMRES iterations (OK), and for the fine mesh, inner FGMRES iterations preconditioned with the multigrid scheme of \cref{sec:mg} (IK) for the 3D five-holes quadruple-pipe problem. The bracketed numbers in the OK and IK columns are the average number of outer FGMRES iterations per BM iteration and average number of inner FGMRES iterations per augmented momentum block solve, respectively. The barrier parameter is initialized at $\mu_0 =200$ on the coarse mesh and $\mu_0=10^{-5}$ on the fine mesh.} 
\label{tab:five-holes-quad-pipe}
\end{table}

\begin{table}[ht]
\small
\centering
\begin{tabular}{l|lll}
\toprule
Strategy & BM & Time taken (s)& Avg.~time taken (s) \\
\midrule
\multicolumn{4}{c|}{Coarse mesh} \\
\midrule
\labelcref{dp:s2} 1 CPU & 6 & 1664.02 & 277.33\\
\labelcref{dp:s2} 4 CPUs & 44 & 4425.71 &100.58\\
\labelcref{dp:s2} 8 CPUs & 37 & 2310.22 &62.44\\
\labelcref{dp:s2} 16 CPUs & 42 & 1844.64& 43.92\\
\labelcref{dp:s2} 32 CPUs & 6 & 208.64 & 34.77\\
\labelcref{dp:s3} 32 CPUs & 7 & 2150.24 & 307.18 \\
\midrule
\multicolumn{4}{c|}{Fine mesh} \\
\midrule
\labelcref{dp:s3} 32 CPUs & 41 & 53,356.56 & 1301.38 \\
\bottomrule
\end{tabular}
\caption{Time measurements of the linear solves of the BM system. On the coarse mesh we measure the number of BM iterations and overall time taken to find the first branch at $\mu=200$. On the fine mesh we measure the number of BM iterations and overall time taken to grid-sequence the first solution starting at $\mu=10^{-5}$.} 
\label{tab:five-holes-quad-pipe-timings}
\end{table}

\begin{figure}[ht]
\centering
\subfloat[Branch 0]{\includegraphics[width = 0.32\textwidth]{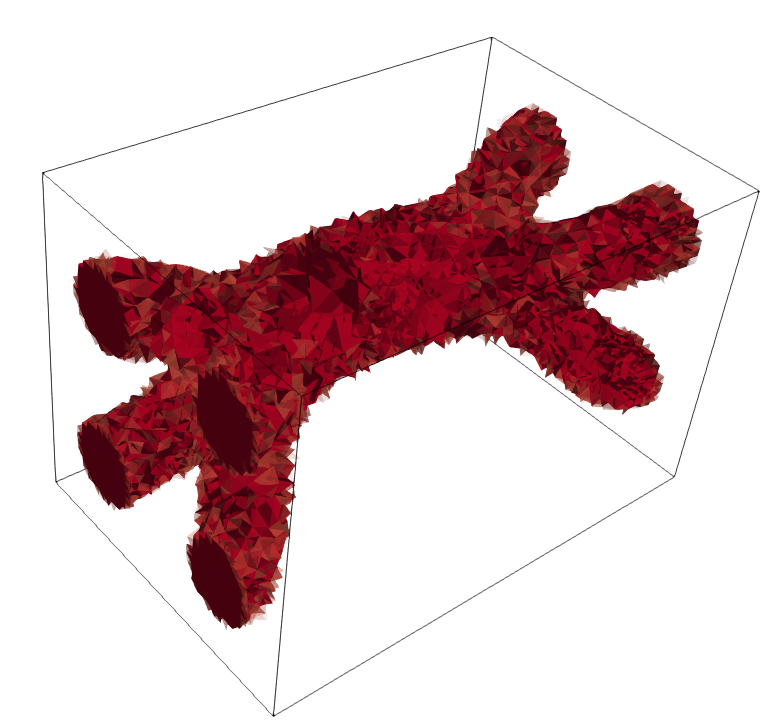}}
\subfloat[Branch 1]{\includegraphics[width = 0.32\textwidth]{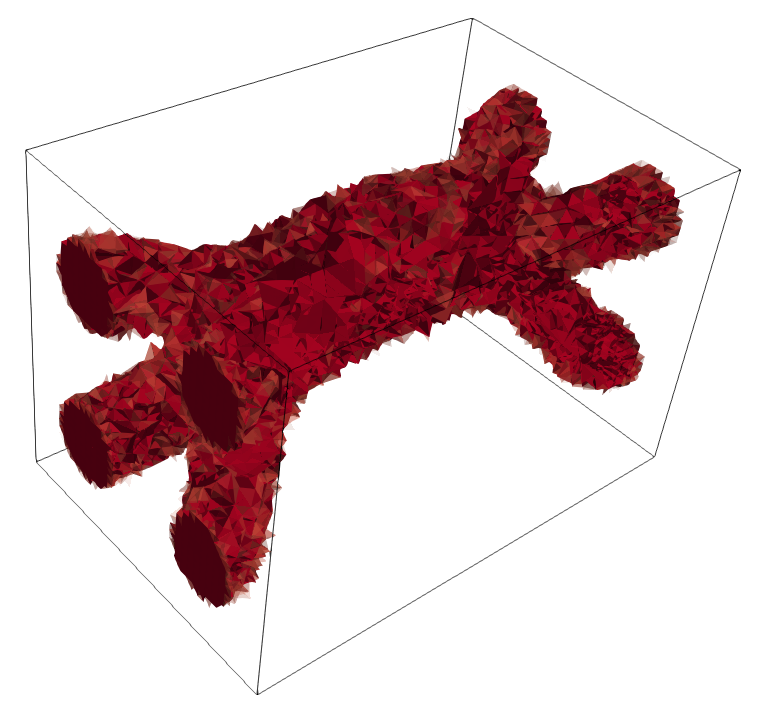}}
\subfloat[Branch 2]{\includegraphics[width = 0.32\textwidth]{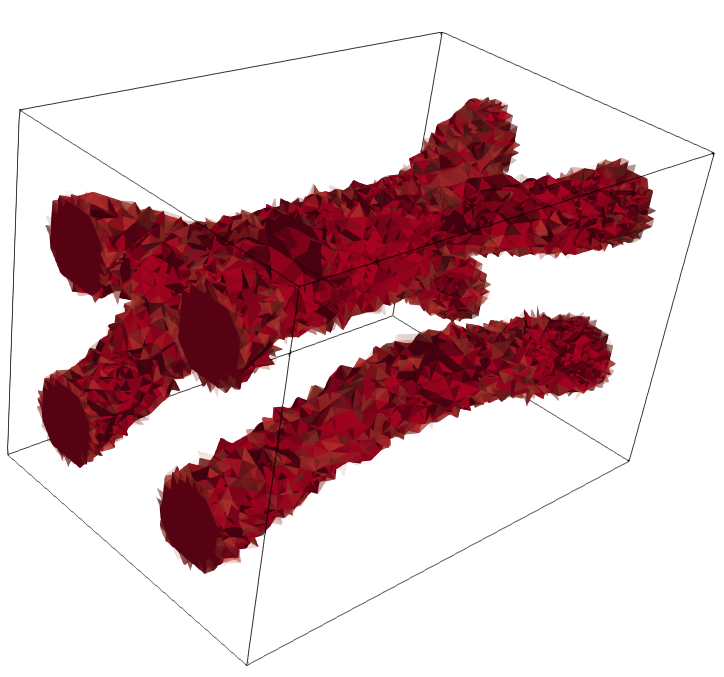}}\\
\subfloat[Branch 3]{\includegraphics[width = 0.32\textwidth]{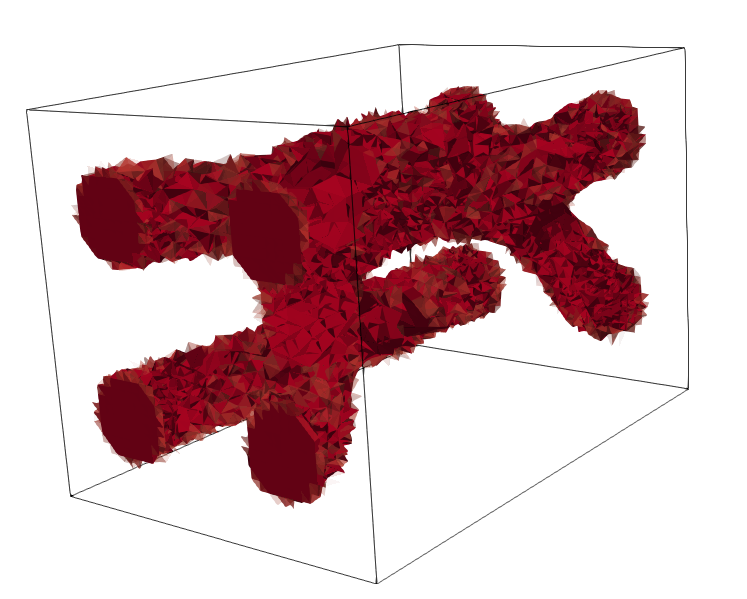}}
\subfloat[Branch 7]{\includegraphics[width = 0.32\textwidth]{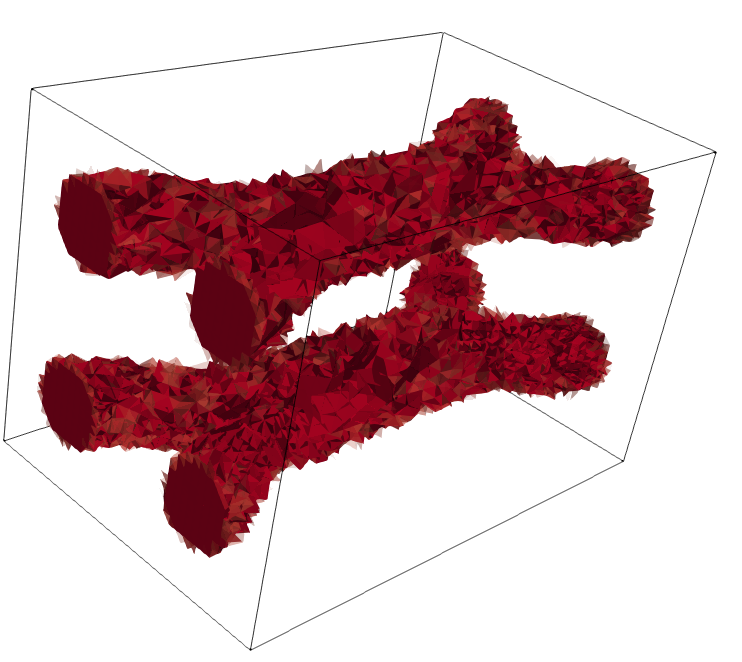}}
\subfloat[Branch 11]{\includegraphics[width = 0.32\textwidth]{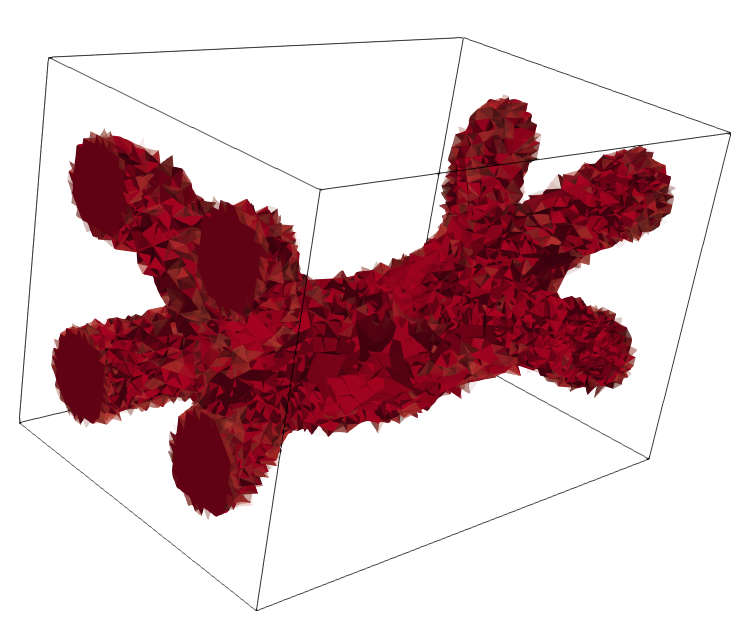}}\\
\subfloat[Branch 13]{\includegraphics[width = 0.32\textwidth]{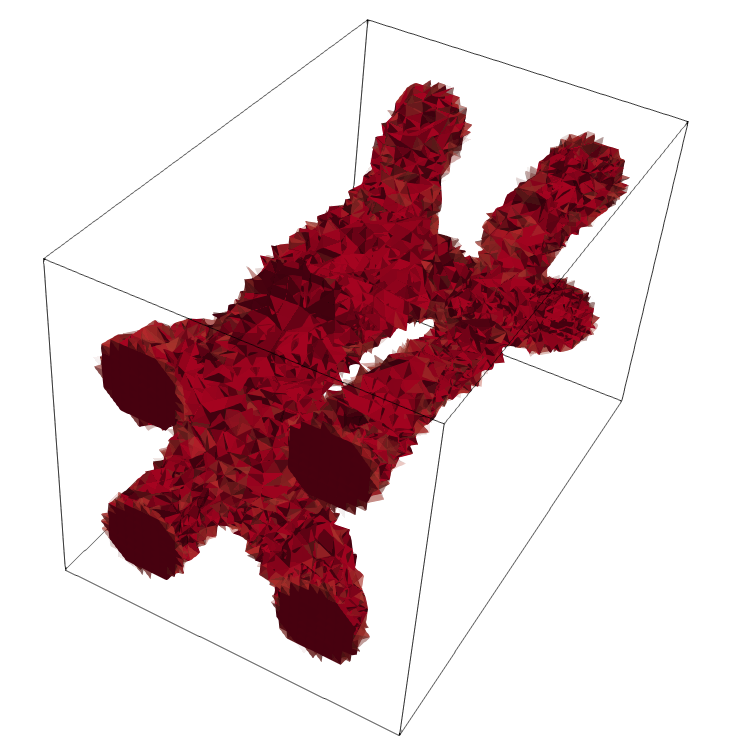}}
\caption{The material distribution of seven solutions to the 3D fives-holes quadruple-pipe optimization problem discretized with 2,014,113 degrees of freedom as discovered by the deflated barrier method and grid-sequenced once. The resulting power dissipation values for branches 0, 1, 2, 3, 7, 11, and 13 are $J_h = 55.02, 54.73, 62.27,  62.22, 59.56, 55.28$, and $62.78$, respectively.}
\label{fig:3d-5-holes-1}
\end{figure}
\begin{figure}[ht]
\centering
\subfloat[Branch 0]{\includegraphics[width = 0.23\textwidth]{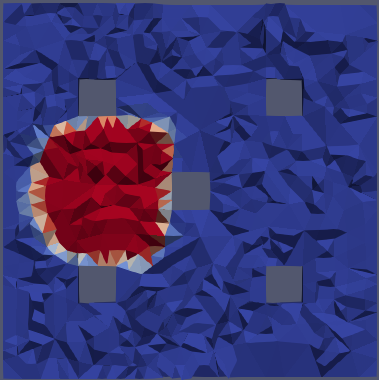}}\;
\subfloat[Branch 1]{\includegraphics[width = 0.23\textwidth]{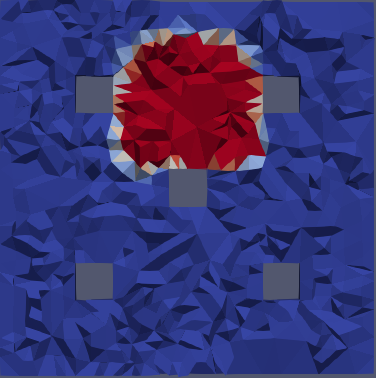}}\;
\subfloat[Branch 2]{\includegraphics[width = 0.23\textwidth]{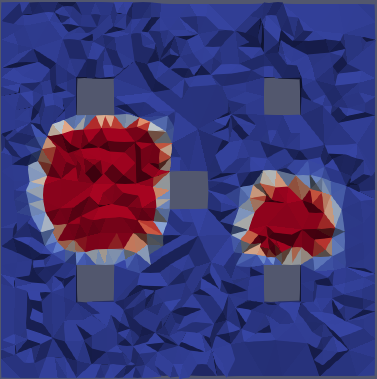}}\;
\subfloat[Branch 3]{\includegraphics[width = 0.23\textwidth]{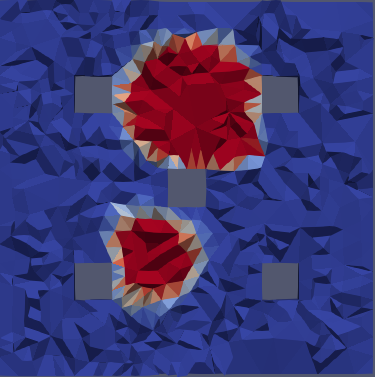}}\\
\subfloat[Branch 7]{\includegraphics[width = 0.23\textwidth]{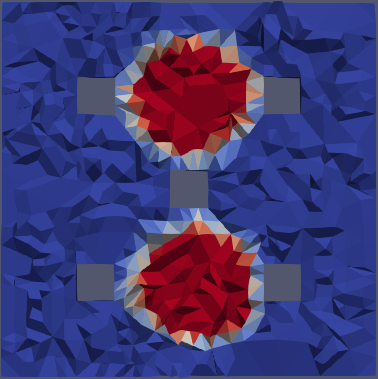}}\;
\subfloat[Branch 11]{\includegraphics[width = 0.23\textwidth]{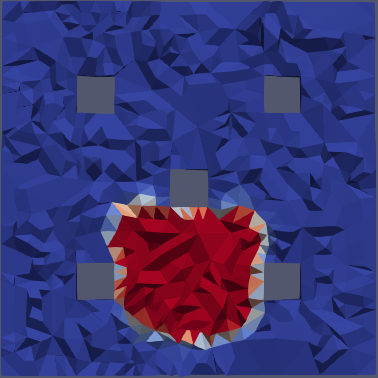}}\;
\subfloat[Branch 13]{\includegraphics[width = 0.23\textwidth]{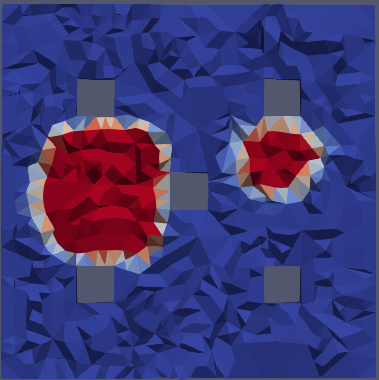}}
\caption{The crinkled cross sections at $x=3/4$ for the discovered solutions of the 3D five-holes quadruple-pipe. The grey regions are part of the five cuboid holes in the box domain. The material distribution has a value of one in the red regions and zero in the blue regions, with intermediate values for the intermediate coloured regions.}
\label{fig:3d-5-holes-slice}
\end{figure}
\begin{figure}[ht]
\centering
\includegraphics[width = 0.6\textwidth]{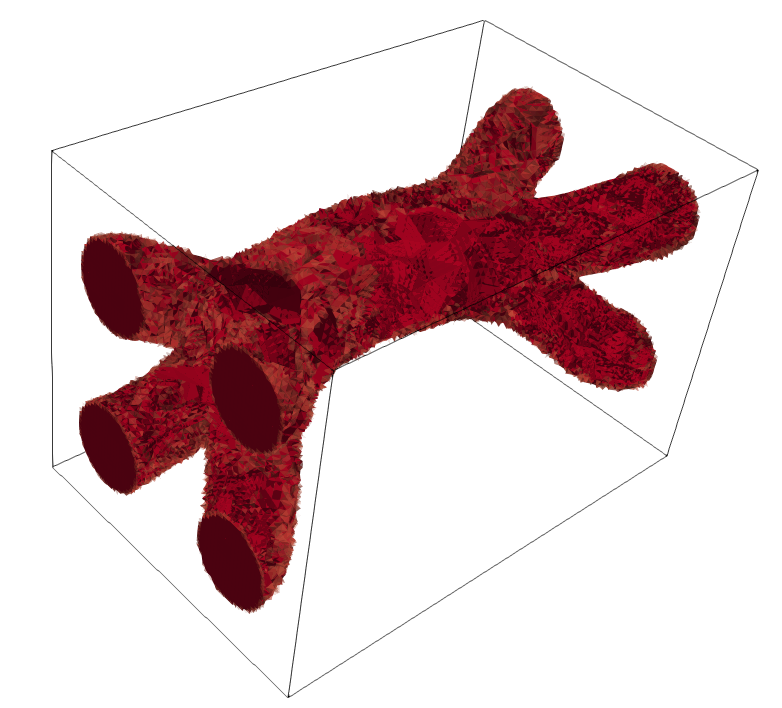}
\caption{The material distribution of branch 0 after grid-sequencing to a second uniform mesh refinement resulting in 15,953,537 degrees of freedom ($J_h = 39.11$). The augmented Lagrangian 3-grid multigrid preconditioner \labelcref{dp:s3} was used.}
\label{fig:3d-5-holes-2}
\end{figure}

\textbf{Code availability:} For reproducibility, the implementation of the deflated barrier method used in this work, as well as scripts to generate the solutions, can be found at \url{https://github.com/ioannisPApapadopoulos/fir3dab/}. The version of the software used in this paper is archived on Zenodo \cite{fir3dab}. 

\newpage 

\section{Conclusions}
In this work we extended the application of the deflated barrier method to discovering multiple three-dimensional solutions of the fluid topology optimization model of Borrvall and Petersson \cite{Borrvall2003}. This was achieved by developing preconditioners for the linear systems that arise in the deflated barrier method. The preconditioning strategy reduces the discretized $4 \times 4$-block matrix solve to an outer FGMRES method and the solve of one diagonal matrix, one block-diagonal matrix (that can be factorized once or quickly solved with multigrid), and an augmented momentum block. Moreover, we develop a geometric multigrid cycle for the augmented momentum block that consists of a specialized relaxation method and a characterization of the active set on coarser levels. We found that the preconditioner is robust to the mesh size and successfully computed three solutions of a 3D cross-channel problem and seven solutions of a 3D five-holes quadruple-pipe problem.


\clearpage
\bibliographystyle{siamplain}
\bibliography{references}
\end{document}